\documentclass[12pt]{elsarticle}

\usepackage{amssymb}
\usepackage{amsmath}
\usepackage{amsfonts}
\usepackage{mathrsfs}
\usepackage{amsthm}

\usepackage{hyperref}
\usepackage{doi}

\usepackage[capitalise]{cleveref}
\usepackage{xspace}
\usepackage{enumitem}
\usepackage{siunitx}
\usepackage{multirow}

\usepackage{a4wide}
\usepackage{todonotes}
\usepackage{booktabs}

\usepackage{subcaption}

\newcommand{\dumux}{DuMu\textsuperscript{x}\xspace}

\newcommand{\calD}{\mathcal{D}}
\newcommand{\calE}{\mathcal{E}}

\newcommand{\calP}{\mathcal{P}}

\newcommand{\calT}{\mathcal{T}}
\newcommand{\calV}{\mathcal{V}}
\newcommand{\calS}{\mathcal{S}}
\newcommand{\calM}{\mathcal{M}}
\newcommand{\calDofs}{\mathcal{L}}

\newcommand{\n}{\mathbf{n}}

\newcommand{\meas}[1]{|{#1}|}

\newcommand{\vel}{\boldsymbol{v}}
\newcommand{\veltest}{\boldsymbol{w}}
\newcommand{\veldisc}{\boldsymbol{v}_h}
\newcommand{\veldisctest}{\boldsymbol{w}_h}
\newcommand{\p}{p}
\newcommand{\ptest}{q}
\newcommand{\pdisc}{p_h}

\newcommand{\PSpace}{\mathbb{P}}

\newcommand{\QSpace}{\mathbb{Q}}

\newcommand{\discspace}{\mathit{V}_h}
\newcommand{\dimspace}[2]{\left[#1\right]^{#2}}
\newcommand{\bubblespace}{\PSpace_{1+b}}

\newcommand{\boxspace}{\PSpace_1}

\newcommand{\bubblebox}[1]{\ensuremath{\dimspace{\bubblespace}{#1} \times \boxspace}}
\newcommand{\bubbleboxname}[2]{\ensuremath{\dimspace{\bubblespace^{\text{#2}}}{#1} \times \boxspace}}

\newcommand{\bubbleboxovtable}[1]{\ensuremath{\begin{array}{c} \dimspace{\bubblespace^{\text{ov}}}{#1} \\ {}\times{} \\ \boxspace\end{array}}}
\newcommand{\bubbleboxnovtable}[1]{\ensuremath{\begin{array}{c} \dimspace{\bubblespace^{\text{nov}}}{#1} \\ {}\times{} \\ \boxspace\end{array}}}
\newcommand{\bubbleboxhytable}[1]{\ensuremath{\begin{array}{c} \dimspace{\bubblespace^{\text{hy}}}{#1} \\ {}\times{} \\ \boxspace\end{array}}}
\newcommand{\bubbleboxfemtable}[1]{\ensuremath{\begin{array}{c} \dimspace{\bubblespace^{\text{fem}}}{#1} \\ {}\times{} \\ \boxspace\end{array}}}

\newcommand{\elem}{E}
\newcommand{\gridFace}{\epsilon}

\theoremstyle{plain}

\newtheorem{proposition}{Proposition}

\newtheorem{definition}{Definition}

\theoremstyle{remark}
\newtheorem{remark}{Remark}
\newtheorem{example}{Example}

\newcommand{\eqbydef}{:=}

\journal{}

\begin{document}

\begin{frontmatter}



\title{Stable and locally mass- and momentum-conservative control-volume finite-element schemes for the Stokes problem}


\author[inst1]{Martin Schneider}

\affiliation[inst1]{organization={Department of Hydromechanics and Modelling of Hydrosystems},
            addressline={{Pfaffenwaldring 61}}, 
            city={Stuttgart},
            postcode={70569}, 
            country={Germany}}

\author[inst2]{Timo Koch}

\affiliation[inst2]{organization={Department of Mathematics, University of Oslo},
            addressline={Postboks 1053, Blindern}, 
            city={Oslo},
            postcode={0316}, 
            country={{Norway}}}

\begin{abstract}
We introduce new control-volume finite-element discretization schemes suitable for solving the Stokes problem. Within a common framework, we present different approaches for constructing such schemes. The first and most established strategy employs a non-overlapping partitioning into control volumes. The second represents a new idea by splitting into two sets of control volumes, the first set yielding a partition of the domain and the second containing the remaining overlapping control volumes required for stability. The third represents a hybrid approach where finite volumes are combined with finite elements based on a hierarchical splitting of the ansatz space. All approaches are based on typical finite element function spaces but yield locally mass and momentum conservative discretization schemes that can be interpreted as finite volume schemes.
We apply all strategies to the inf-sub stable MINI finite-element pair. Various test cases, including convergence tests and the numerical observation of the boundedness of the number of preconditioned Krylov solver iterations, as well as more complex scenarios of flow around obstacles or through a three-dimensional vessel bifurcation, demonstrate the stability and robustness of the schemes.
\end{abstract}


\end{frontmatter}


\section{Introduction}
\label{sec:intro}

Let $\Omega \subset \mathbb{R}^d$, $d = 2,3$ be a bounded Lipschitz domain with boundary $\partial \Omega$. The incompressible stationary Stokes equations are given by
\begin{subequations}
\begin{align}
   -\operatorname{div}\left( 2 \mu \boldsymbol{D}(\vel) - \p \boldsymbol{I} \right) &= \boldsymbol{f} && \text{in} \;\Omega, \label{eq:MomentumBalance} \\
   \operatorname{div} \vel &= 0 && \text{in} \;\Omega, \label{eq:MassBalance}
\end{align}
\label{eq:stokes}
\end{subequations}
with velocity $\vel$, pressure $p$, dynamic fluid viscosity $\mu$, a momentum source term $\boldsymbol{f}$ (e.g.\ modeling gravity or other body forces), and the symmetric velocity gradient $\boldsymbol{D}(\vel) = \frac{1}{2}\left(\nabla \vel + \nabla^T \vel\right)$, subject to the boundary conditions
\begin{subequations}
\begin{align}
   \vel &= \vel_D && \text{on} \;\partial\Omega_D, \label{eq:DirichletBC} \\
   -\left( 2 \mu \boldsymbol{D}(\vel) - \p \boldsymbol{I} \right)\boldsymbol{n} &= \boldsymbol{t}_N && \text{on} \;\partial\Omega_N, \label{eq:NeumannBC}
\end{align}
\label{eq:stokesBC}
\end{subequations}
where $\partial\Omega_D \cup \partial\Omega_N = \partial \Omega$, $\boldsymbol{n}$ is the outward-oriented unit normal vector on $\partial \Omega$, and $\vel_D$, $\boldsymbol{t}_N$ the given boundary data.

When constructing numerical solvers for \cref{eq:stokes,eq:stokesBC}, it is well known that discretization schemes have to be carefully constructed to obtain physically consistent and accurate solutions. For example, na\"{i}ve finite-volume (FV) or finite-difference (FD) methods with both velocity and pressure degrees of freedom placed at the same locations suffer from checkerboard oscillations in the pressure field even though the velocity field may be well approximated \citep{Langtangen2002}. This is why a staggered arrangement of degrees of freedom has been proposed in the literature, yielding so-called staggered-grid finite-volume schemes, see e.g.~\cite{harlow1965a,ferziger2020computational,schneider2020}. Alternative approaches introduce stabilization terms to prevent checkerboard oscillations \cite{Brezzi1984,Hughes1986}.

For the construction of finite-element schemes, it is well known that velocity and pressure spaces cannot be chosen independently but instead so-called LBB-stable pressure-velocity pairs, i.e.\ pairs satisfying the Ladyzhenskaya-Babu{\v{s}}ka-Brezzi condition~\cite{BrezziFortin1991}, have to be chosen. Examples of lowest-order schemes are, for example, the MINI element~\cite{Arnold1984}, which uses an $H^1$-conforming velocity space enriched with so-called bubble functions; 
the Crouzeix–Raviart (CR) element \cite{Crouzeix1973}, which uses a non-conforming $\PSpace_1$ basis for the velocity space on simplicial meshes; or the Rannacher-Turek (RT) element \cite{Rannacher1992}, using a non-conforming basis on hexahedral grids. These schemes generally are conservative if and only if piece-wise constant functions $\PSpace_0$ can be described by the test function space. This holds for the CR and RT pressure spaces, yielding local mass conservation, but not for the MINI element. Local momentum conservation\footnote{without additional post-processing of the results, see e.g.~\citep{Hughes2000}} does generally not hold for any of these finite-element schemes.

Local conservation is obtained by construction with so-called
control-volume finite-element schemes (CVFE) which try to combine the flexibility of finite elements with the transparent local conservation properties and simplicity of finite-volume schemes. CVFE schemes can be interpreted both as finite-volume schemes, using particular discrete local gradient operators, and as Petrov-Galerkin finite-element schemes~\cite{Bank1987Box}, using particular combinations of trial and test function spaces. After the development of the classical CVFE scheme, a vertex-centered finite-volume scheme based on element-wise linear polynomial and globally conforming function approximations, also known as the Box method~\cite{Winslow1966,Bank1987Box}, CVFE schemes have since been extended for solving the Stokes equations \labelcref{eq:stokes}. Such extensions include an equal-order ansatz with stabilization terms~\cite{Prakash1985,Eymard2006,Li2009,Quarteroni2011}; face-centered finite volumes with CR-$\PSpace_0$ or RT-$\PSpace_0$ finite elements with~\cite{Zhang2015} and without stabilization~\cite{Chou1997,Chou1998,ye2001relationship,Ye2006}; BDM-$\PSpace_0$ finite elements~\cite{Brezzi1985,Wang2010}; or stabilized Q1-Q0~\cite{Zhang2019} with vertex-centered finite volumes. Recently, CVFE schemes based on element-wise second-order Lagrange polynomial spaces ($\PSpace_2$) have been presented in~\cite{Chen2017,Zhang2022}.

To the best of our knowledge, the schemes presented in these works suffer from one or more of the following issues: (1) The scheme is not locally mass- and momentum-conservative or just one conservation property is satisfied. For example, schemes based on equal-order stabilized methods~\citep{Prakash1985,Eymard2006,Li2009,Quarteroni2011} usually compromise local mass conservation due to the added stability term. (The scheme may still be locally conservative, however with a modified flux function.) (2) Local conservation may be satisfied weakly but not strongly in the discrete form and is therefore subject to the typical numerical approximation errors. This is particularly important when using the approximated velocity field in a tracer transport simulation, where non-physical perturbations can be introduced by non-conservative velocity fields \citep{Chen2002,Jenny2003,Sun2006}.
(3) The scheme construction does not straight-forwardly generalize from 2D to 3D~\cite{Chen2017,Zhang2022,yang2023} (constructing non-overlapping control-volume partitions is difficult, cf.~\cref{subsec:cvs}). 
(4) The scheme does not satisfy the Ladyzhenskaya-Babu{\v{s}}ka-Brezzi (LBB) conditions~\cite{BrezziFortin1991} when using the symmetric part of the velocity gradient. For example, schemes based on non-conforming finite-element spaces, such as \citep{Chou1997,ye2001relationship,Ye2006}, are only stable (without additional stabilization terms) if the full velocity gradient is used to model the viscous term in the Stokes equations.

In this work, two methods are presented to overcome these issues: an overlapping control-volume scheme and a hybrid CVFE-FEM scheme.

In literature, control-volume overlap usually refers to control volumes for different primary variables overlapping. For instance, in the staggered-grid scheme \cite{harlow1965a} the control volumes for the different velocity components are overlapping. This approach has also been used for other mixed problems, e.g.~\citep{Cai1997,Chou2001}, where the control volumes constructed for the pressure variable overlap with those constructed for the velocity. In \cite{Onate1994,Cardiff2021} all control volumes of the same variable overlap with neighboring volumes.
The strategy presented in this work differs fundamentally from these existing approaches. Here, overlapping control volumes are constructed only for a subset of degrees of freedom needed for LBB stability and in a way that all control-volume faces are always fully contained within elements.

The hybrid CVFE-FEM approach is related to the work presented in \cite{chen2010new} for elliptic equations. The idea is to combine CVFE and FEM approaches based on a hierarchical splitting of the basis function ansatz space.

In \cref{sec:preliminaries,subsec:cvfeschemes}, we provide a general framework for the construction of CVFE schemes. We present three related CVFE schemes for the Stokes equations that are both stable and strongly conserve mass and momentum locally, based on the MINI element~\cite{Arnold1984}. The local conservation property is satisfied by construction. In \cref{sec:numeric}, evidence for the stability of the methods will be provided in the form of numerical experiments of two kinds: grid convergence studies with analytical reference solutions, and the observation of bounded preconditioned linear solver iterations strongly indicating the stability of the discrete form. We conclude with two numerical examples of Stokes flow in complex domains.

\section{Preliminaries}
\label{sec:preliminaries}

For a bounded domain $\Omega \subset \mathbb{R}^d$, $d = 2,3$, let $L^2(\Omega)$ denote the space of square integrable functions, and $H^k(\Omega)$, $k\geq1$ the Sobolev space of functions on $\Omega$ with weak derivatives up to order $k$. Boldface variants correspond to vector-valued functions of dimension $d$. Homogeneous boundary conditions on a part of the boundary $\Gamma \subseteq \partial \Omega$ are indicated by a subscript $0$ as in the example $H_0^1(\Omega) := \lbrace u \in H^1(\Omega) \,\vert\, u = 0 \text{ on } \Gamma \rbrace$.
The dual space of $H_0^1(\Omega)$ with respect to the $L^2$ inner product is denoted by $H^{-1}(\Omega)$. The $L^2$ inner product is denoted by $(\cdot,\cdot)$.

\subsection{Weak form of the Stokes problem}
A weak problem formulation of \cref{eq:stokes} with mixed homogeneous boundary conditions, \cref{eq:stokesBC}, is given by:

Find $(\vel, p) \in \boldsymbol{H}^1_0 \times L^2$ such that
\begin{subequations}
\begin{align}
    (2\mu\boldsymbol{D}(\vel), \boldsymbol{D}(\veltest)) - (p, \operatorname{div} \veltest) &= (\boldsymbol{f}, \veltest) 
    && \forall \veltest \in \boldsymbol{H}^1_0, \\
   (\operatorname{div} \vel, q) &= (g, q)  && \forall \ptest \in L^2,
\end{align}
\label{eq:stokes_weak}
\end{subequations}
where we generalized to allow for non-homogeneous right-hand sides in the second equation.
Well-posedness of the continuous problem in this setting follows from classical abstract saddle-point problem theory~\cite{Brezzi1974}.

\begin{remark}
In the context of control-volume finite-element schemes, as discussed in the following, we typically choose different discrete function spaces to look for the discrete trial functions ($\boldsymbol{v}_h$, $p_h$) and test functions ($\veldisctest$, $q_h$), that is the Petrov-Galerkin approach.
\end{remark}

\subsection{Spatial discretization}
\label{sec:spatialdisc}
In this work, we present a control-volume finite-element (CVFE) discretization framework. We make use of the following grid discretization definition, which is based on \cite{schneider2018comparison,schneider2021coupling}:

\begin{definition}[Grid discretization]
\label{def:griddisc}
The tuple \mbox{$\calD \eqbydef (\calM, \calS, \calT,\calE,\calP,\calV,\calDofs)$} denotes the grid discretization, in which
  \begin{enumerate}[label=(\roman*)]
  \item $\calM$ is the set of primary grid elements such that
    \mbox{$\overline{\Omega}= \cup_{\elem \in \calM} \overline{\elem}$}.
    For each element $\elem\in\calM$, $\meas{\elem}>0$ and $\mathbf{x}_\elem$ denote its volume and barycenter.
  \item $\calS$ is the set of faces such that each face $\gridFace$ is a $(d-1)$-dimensional hyperplane
        with measure
        \mbox{$\meas{\gridFace}>0$}.
        For each primary grid element $\elem \in \calM$, $\calS_\elem$ is the subset of
        $\calS$ such that \mbox{$\partial \elem  = \cup_{\gridFace \in\calS_\elem}{\overline{\gridFace}}$}. Furthermore, $\mathbf{x}_\gridFace$ denotes
        the barycenter and $\n_{\elem,\gridFace}$ the unit vector that is normal to
        $\gridFace$ and outward to $\elem$.
  \item $\calT$ is the set of control volumes (CVs) such that
    \mbox{$\overline{\Omega}= \cup_{K \in \calT} \overline{K}$}.
    For each control volume $K\in\calT$, $\meas{K}>0$ denotes its volume. The element-local contributions of a CV, i.e.\ $K \cap E$, are called sub-control volumes.
 \item
    $\calE$ is the set of faces such that each face $\sigma$ is a $(d-1)$-dimensional hyperplane
    with measure
    \mbox{$\meas{\sigma}>0$}.
    For each $K \in \calT$, $\calE_K$ is the subset of
    $\calE$ such that \mbox{$\partial K  = \cup_{\sigma \in\calE_K}{\overline{\sigma}}$}. Furthermore, $\mathbf{x}_\sigma$ denotes
    the face evaluation points and $\n_{K,\sigma}$ the unit vector that is normal to
    $\sigma$ and outward to $K$. The element-local contributions, i.e.\ $\sigma \cap E$, are called sub-control-volume faces. 
  \item
    $\calP \eqbydef \lbrace \mathbf{x}_K\rbrace_{K \in \calT}$ is the set of points associated each  with a control volume $K$ such that
    $\mathbf{x}_K\in K$ and $K$ is star-shaped with respect to $\mathbf{x}_K$.
  \item
    $\calV$ is the set of vertices of the grid, corresponding to the corners of the primary grid elements $\elem \in \calM$.
  \item
    $\calDofs \subset \Omega$ is the set of points associated with the locations of the degrees of freedom of a nodal basis. 
  \end{enumerate}
\end{definition}

For an element-wise consideration, we define for $\mathcal{X} \in \lbrace \calT, \calE \rbrace$ the sets $\mathcal{X}_E \eqbydef \lbrace \chi \in \mathcal{X} \;|\; \meas{\chi \cap E} > 0 \rbrace$, for each $E \in \calM$, and for $\mathcal{X} \in \lbrace \calM, \calS \rbrace$ the sets $\mathcal{X}_K \eqbydef \lbrace  \chi \in \mathcal{X} \;|\; \meas{K \cap E} > 0 \rbrace$, for each $K \in \calT$. 
In the following, $\calT$ is not necessarily a partition of $\Omega$. Control volumes $K$ may overlap. However, it is assumed that there exists a subset $\calT_{\Omega} \subseteq \calT$ such that $\overline{\Omega} = \cup_{K \in \calT_{\Omega}} \overline{K}$ and $\meas{K\cap L} = 0$ for all $K \not = L$,
that is, $\calT_{\Omega}$ is a partition of $\Omega$.

For the Stokes problem, the control volumes chosen for the pressure degrees of freedom may be different than those chosen for the velocity degrees of freedom. This is indicated by a superscript, e.g.\ $\calT^\p$ and $\calT^{\vel}$. These superscripts are also used to distinguish other quantities, i.e.\ we write $(\cdot)^\p$, $(\cdot)^{\vel}$.  

\subsection{Discrete solution representation and nodal bases}

In the following, the discrete velocity and pressure solutions for the Stokes problem are expressed in terms of nodal bases.


\begin{definition}[Nodal basis]
A set of functions $ \Phi(\calDofs) = \lbrace \varphi_i \; | \; i = 1,\dotsc, |\calDofs| \rbrace$, where each function $\varphi_i$ is uniquely identified with the degrees of freedom located at $\mathbf{x}_i \in \calDofs$, 
is called a piece-wise differentiable nodal basis if for all $\varphi_i \in \Phi$:
\begin{equation}
\varphi_i \arrowvert_E  \in C^1(E)\quad \forall E \in \calM  \quad \text{and}\quad \varphi_i\arrowvert_E (\mathbf{x}_{j}) = \delta_{ij}  \quad \forall E \in \calM, \forall \mathbf{x}_j \in \calDofs.
\label{eq:localBasisUnity}
\end{equation}
\end{definition}

As before, we use the superscript notation to distinguish bases for velocity and pressure, $\Phi^{\vel}$, $\Phi^{p}$. 

Let $\vel_i:= \vel(\mathbf{x}_i), \; \forall \mathbf{x}_i \in \calDofs^{\vel}$, $p_i:= p(\mathbf{x}_i), \; \forall \mathbf{x}_i \in \calDofs^{p}$, denote point-evaluations of the functions $\vel$, $p$ at some appropriate locations $\calDofs^{\vel}$, $\calDofs^{p}$ (often coinciding with control-volume centers). Then, the discrete velocity and pressure solutions are defined in terms of their representation in some choice of nodal basis $\Phi^{\vel}(\calDofs^{\vel})$, $\Phi^{p}(\calDofs^{p})$, with coefficients $\vel_i$, $p_i$, such that
\begin{equation}
\vel_h := \sum_{i = 1}^{|\calDofs^{\vel}|} \vel_i \varphi^{\vel}_i, \quad \varphi^{\vel}_i \in \Phi^{\vel}(\calDofs^{\vel}), \quad\quad 
\p_h:= \sum_{i = 1}^{|\calDofs^{\p}|} \p_i \varphi^{\p}_i, \quad \varphi^{\vel}_i \in \Phi^{p}(\calDofs^{p}).
\label{eq:discSolAnsatz}
\end{equation}

\begin{remark}
The discrete solutions $\vel_h$ and $\p_h$ have the same regularity as the basis functions and since the basis functions are defined to be $C^1$ within elements, gradients can be computed \emph{within} elements. In what follows, this property will be essential for the construction of flux approximations on sub-control-volume faces.
\end{remark}

\subsection{Control-volume test function space}
The following definition is useful when discussing the relationship between finite volumes and finite elements in the context of control-volume finite-element schemes.

\begin{definition}[Control-volume function space]
\label{def:cvfs}
Given a set of control volumes $\widetilde{\calT} \subseteq \calT$,
and the functions $C_K \in L^2 (\widetilde{\Omega})$ with $ \widetilde{\Omega} := \cup_{K\in\widetilde{\calT}} K \subseteq \Omega$ such that
\begin{equation}
C_K(\mathbf{x}) := \begin{cases} c_K \in \PSpace^0(K) &\mathbf{x} \in K, \\ 0 &\mathbf{x} \notin K,\end{cases}
\end{equation}
the discrete function space
\begin{equation}
B_h(\widetilde{\calT}) = \left\lbrace q_h \in L^2 (\widetilde{\Omega}): \; q_h = \sum_{K\in \widetilde{\calT}} C_K \right\rbrace,
\end{equation}
contains all functions $q_h$ that are sums of piece-wise constant functions on control volumes $K$ and the dimension of space $B_h(\widetilde{\calT})$ equals the number of control volumes in $\widetilde{\calT}$.
\end{definition}

\begin{proposition}
\label{prop:fem-fv}
Given a space $V(\Omega) \subset L^2(\Omega)$, and $r \in V(\Omega)$, the following equivalence holds
\begin{equation*}
    \int_\Omega r q_h \;\text{d}x = 0 \quad \forall q_h \in B_h(\widetilde{\calT}) \quad \iff \quad \int_K r \;\text{d}x = 0 \quad \forall K \in \widetilde{\calT}.
\end{equation*}
\end{proposition}
\begin{proof}
For each $q_h \in B_h(\widetilde{\calT})$ it holds that 
\begin{equation*}
\int_\Omega r q_h \;\text{d}x = \int_\Omega r \left( \sum_{K \in \widetilde{\calT}} C_K \right) \;\text{d}x  =  \sum_{K \in \widetilde{\calT}}  \int_\Omega r C_K  \;\text{d}x = \sum_{K \in \widetilde{\calT}} c_K \int_K r  \;\text{d}x,
\end{equation*}
from which the equivalence follows directly. 
\end{proof}

\Cref{prop:fem-fv} makes clear how control volumes schemes can be interpreted as finite-element schemes with a particular choice for the test function space. Note that a control volume in $\widetilde{\calT}$ may be fully or partially overlapping with other control volumes in $\widetilde{\calT}$.

\section{Control-volume finite-element (CVFE) schemes}
\label{subsec:cvfeschemes}
The idea of the CVFE schemes presented here is the construction of control volumes around the degree of freedom locations (of a nodal basis) such that the corresponding element-local control-volume face contributions, i.e.\ sub-control-volume faces, are embedded within elements. In other words, this means that sub-control-volume faces never coincide with the element boundary $\partial E$. This ensures that normal derivatives of the approximated functions on faces can be computed using the element-local basis-function gradients.

Following the idea of finite-volume methods, the discretization of \cref{eq:MomentumBalance,eq:MassBalance} is derived by integrating over the control volumes $K$ and applying the divergence theorem:
\begin{subequations}
\label{eqs:discFVBalances}
\begin{align}
   \int_{\partial K} \left[ -2 \mu \boldsymbol{D}(\veldisc) + \pdisc \boldsymbol{I} \right] \cdot \n \, \text{d}A  &= \int_K \boldsymbol{f} \, \text{d}x, \quad \forall K \in \calT^{\vel}, \label{eq:DiscMomentumBalance} \\
   \int_{\partial K} \veldisc \cdot \n \, \text{d}A  &= \int_K q \, \text{d}x, \quad \forall K \in \calT^{\p}, \label{eq:DiscMassBalance}
\end{align}
\end{subequations}
where $\veldisc$ and $\pdisc$ denote the discrete velocity and pressure solution, defined in \cref{eq:discSolAnsatz}. 
Making use of \cref{def:griddisc}, each term can be split into its element-local contribution:
\begin{subequations}
\label{eqs:discEleBalances}
\begin{align}
  \sum_{E \in \calM^{\vel}_K} \sum_{\sigma \in \calE_K} \int_{\sigma \cap E} \left[ -2 \mu \boldsymbol{D}(\veldisc) + \pdisc \boldsymbol{I} \right] \cdot \n \, \text{d}A  &= \sum_{E \in \calM^{\vel}_K} \int_{K \cap E} \boldsymbol{f} \, \text{d}x, \quad \forall K \in \calT^{\vel}, \label{eq:DiscEleMomentumBalance} \\
  \sum_{E \in \calM^{\p}_K} \sum_{\sigma \in \calE_K} \int_{\sigma \cap E} \veldisc \cdot \n \, \text{d}A  &= \sum_{E \in \calM^{\p}_K} \int_{K \cap E} q \, \text{d}x, \quad \forall K \in \calT^\p. \label{eq:DiscEleMassBalance}
\end{align}
\end{subequations}
This allows an element-local calculation of integrals which is particularly convenient for element-centered stiffness matrix assembly algorithms. 
A discrete system of equation, in terms of the given degrees of freedom, is derived by inserting the discrete solution \cref{eq:discSolAnsatz} into the flux balances \cref{eqs:discEleBalances}. 

\begin{remark}
An alternative way to arrive at \cref{eqs:discFVBalances} is to multiply with test functions $\veldisctest \in \boldsymbol{B}_h(\calT^{\vel})$ and $q_h \in B_h(\calT^{p})$, integrate, reformulate as in \cref{prop:fem-fv}, and apply the divergence theorem. This alternative derivation makes clear the interpretation of CVFE schemes as Petrov-Galerkin finite-element schemes.
\end{remark}

In the following sections, we discuss how the choice of control volumes leads to two different discretization schemes that are both based on the formulation in \cref{eqs:discEleBalances}. A third scheme is based using a CVFE scheme on a subspace while treating the complementary subspace with a standard finite-element method. Concrete examples of all three schemes are shown \cref{fig:control-volumes}. However, the basic ideas are first presented in a more generic fashion. 

\subsection{Non-overlapping CVFE schemes}

\begin{figure}
    \centering
    \includegraphics[width=0.7\textwidth]{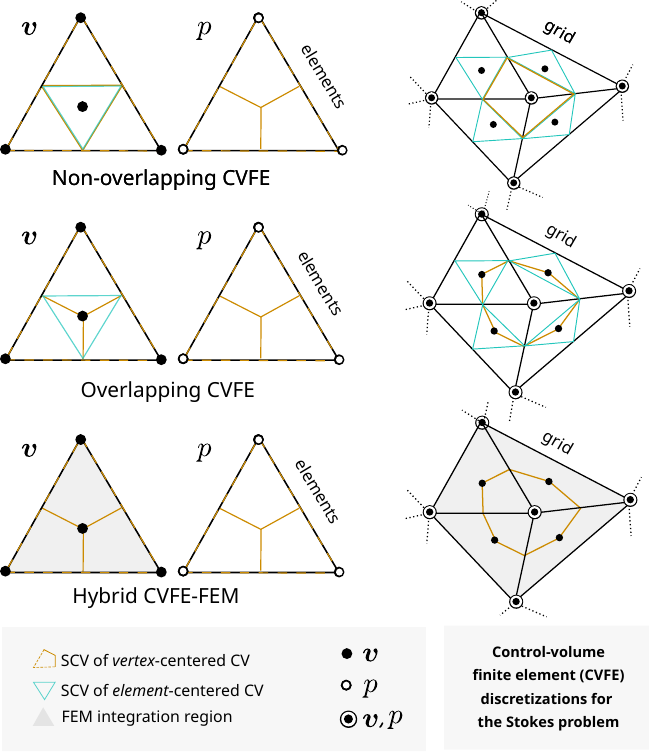}
    \caption{\textbf{CVFE discretizations for the Stokes poblem.} The control volume construction and the degrees of freedom for non-overlapping (top row), overlapping (middle row), and hybrid CVFE-FEM schemes (bottom row) when using the MINI element are shown. The brown vertex-centered ($\vel$, $p$) sub-control volumes of a mesh form a disjoint partition of the domain $\Omega$. Turquoise element-centered control volumes are additional ``overlapping'' control volumes that only contain a strict subset of each element $E \in \calM$ of the mesh. The corner vertices of turquoise control volumes coincide with the face centers of $E$. The figures on the right show the velocity control volumes around degrees of freedom. The control volumes for pressure are not shown but coincide with the brown control volumes when using the overlapping or hybrid approach. The construction trivially generalizes to $3$ dimensions for overlapping and hybrid CVFE-FEM schemes.}
    \label{fig:control-volumes}
\end{figure}

Control volumes for finite-volume schemes are typically chosen such that they form a partition of the domain $\Omega$, i.e.\ $\calT_{\Omega} = \calT$. This choice results in a \emph{non-overlapping CVFE scheme}.

\begin{example}[Vertex-centered $\PSpace_1$- or $\QSpace_1$-CVFE]
\label{example:box}
The probably best-known CVFE scheme, also known as the Box method~\cite{Bank1987Box}, is the non-overlapping vertex-centered method based on a conforming piece-wise linear (or bilinear) Lagrangian trial function space and control volumes $K \in \calT_\Omega = \calT$ are constructed (in 2D) around the vertices $\calV$ by connecting element centers $\mathbf{x}_E$ with element face centers $\mathbf{x}_\epsilon$ of all faces and elements neighboring the vertex. This construction is shown in \cref{fig:control-volumes} for the pressure discretization.
The corresponding test functions are taken in $B_h(\calT)$.
In this work, the $\PSpace_1$-CVFE scheme is used for the pressure subspace for all presented schemes.
\end{example}

For pure cell-centered or vertex-centered schemes, a partition $\calT_{\Omega} = \calT$ is easily found. However, for higher-order polynomial spaces, or for bubble-enriched spaces, creating such a partition for all considered control volumes is not straightforward. This is particularly the case for $d=3$. For a two-dimensional triangular mesh, an example of non-overlapping control volumes for a given choice of degrees of freedom is presented in \cref{subsec:cvs}. 

\subsection{Overlapping CVFE schemes}

A more flexible approach is the construction of partially or fully \emph{overlapping} control volumes. A prominent example of finite-volume methods using overlapping control volumes is given by the so-called staggered grid finite-volume methods, where control volumes associated with different velocity components are overlapping. Discretizing the domain in terms of overlapping control volumes gives rise to \emph{overlapping CVFE schemes}.
To the best of our knowledge, this concept has not yet been explored in the context of CVFE by using a splitting into subsets of control volumes, as explained in the following.

We introduce the idea of using the disjoint splitting
\begin{equation}
    \calT = \calT_{\Omega} \cup \left(\calT \setminus \calT_{\Omega}\right),
\end{equation}
where $\calT_{\Omega}$ denotes the subset of $\calT$ which forms a partition of $\Omega$ (as introduced in \cref{sec:spatialdisc}). The remaining control volumes $\calT \setminus \calT_{\Omega}$ are constructed around the locations of associated degrees of freedom and are allowed to be overlapping with $\calT_{\Omega}$. Thinking of Petrov-Galerkin schemes, this means that the piece-wise constant basis functions have overlapping support but are linearly independent. An example of such overlapping volumes for a second-order scheme (velocity space) is discussed in \cref{subsec:cvs}. An advantage of such overlapping CVFE schemes is that they are simple to implement and control volume construction is more flexible compared to the non-overlapping approach. 

\begin{remark}
Since $\calT_{\Omega}$ is a partition of $\Omega$, schemes based on such construction are locally conservative on the control volumes $K \in \calT_{\Omega}$ (see \cref{sec:conservation}). Balances over the remaining control volumes $K \in \calT \setminus \calT_{\Omega}$ constitute additional conditions on the discrete solution but do not affect the conservation property.
\end{remark}

\subsection{Hybrid CVFE-FEM schemes}
\label{subsec:hcvfeschemes}
Another promising approach, introduced in~\cite{chen2010new}, is based on a hierarchical splitting of the discrete solution space
\begin{equation}
\discspace = \discspace^\downarrow \otimes \discspace^\uparrow.
\label{eq:hierarchicalsplitting}
\end{equation}
This splitting is done such that a domain partitioning into control volumes $\calT_\Omega$ can be easily constructed around degrees of freedom associated with the basis of $\discspace^\downarrow$, whereas degrees of freedom associated with the basis of the space $\discspace^\uparrow$ are treated with a classical finite-element approach without the need for control volumes. This means that different test function spaces are chosen for the two subspaces.

\begin{remark}
In general, we may choose any hierarchical splitting, \cref{eq:hierarchicalsplitting}. However, to maintain the approximation order, the space $\discspace^\downarrow$ should be chosen such that any $\veldisc^\downarrow \in \discspace^\downarrow$ can be described by the finite-volume approach. Therefore, typically  $\discspace^\downarrow$ is chosen as the space of piece-wise linear functions, and the corresponding test space in the CVFE scheme is $B_h(\calT_\Omega)$.
\end{remark}
As an illustration of this concept, let us consider the discrete velocity space and the Stokes momentum balance equation, \cref{eq:MomentumBalance}.
Around the locations of degrees of freedom related to $\discspace^\downarrow$, control volumes $\calT_\Omega$ are given such that \cref{eq:DiscMomentumBalance} can be enforced for all such volumes:
\begin{equation}
   \int_{\partial K} \left[ -2 \mu \boldsymbol{D}(\veldisc) + \pdisc \boldsymbol{I} \right] \cdot \n \, \text{d}A  = \int_K \boldsymbol{f} \, \text{d}x, \quad \forall K \in \calT_\Omega^{\vel}.\label{eq:DiscMomentumBalanceSplitting}
\end{equation}
Conditions for the remaining degrees of freedom are gained by using a standard finite-element approach:
\begin{equation}
    (2\mu\boldsymbol{D}(\veldisc), \boldsymbol{D}(\veldisctest^\uparrow)) - (p, \operatorname{div} \veldisctest^\uparrow) = (\boldsymbol{f}, \veldisctest^\uparrow), 
    \quad \forall \veldisctest^\uparrow \in \discspace^\uparrow.
\label{eq:DiscMomentumBalanceFemSplitting}
\end{equation}

By satisfying these conditions, the discrete solution $\veldisc = \veldisc^\downarrow + \veldisc^\uparrow$ satisfies the flux balances, \cref{eq:DiscMomentumBalanceSplitting}, for all $K \in \calT_\Omega^{\vel}$, and the FEM-type conditions, \cref{eq:DiscMomentumBalanceFemSplitting}, related to $\veldisc^\uparrow$. In particular, this means that the discrete solution satisfies the discrete conservation equation on each control volume and the scheme is locally conservative.

\begin{remark}
Again, an alternative procedure arriving at \cref{eq:DiscMomentumBalanceSplitting} is multiplication with a test function $\veldisctest^{\downarrow} \in \boldsymbol{B}_h (\calT_\Omega)$, reformulation as in \cref{prop:fem-fv}, and application of the divergence theorem. Hence, the form of \cref{eq:DiscMomentumBalanceSplitting} implies the use of piece-wise constant test functions. Also, note that in contrast to \cref{eq:DiscMomentumBalanceFemSplitting} integration by parts is not applicable since $\boldsymbol{B}_h (\calT_\Omega)$ only contains piece-wise constants.
\end{remark}

In this work, we apply this approach to a specific inf-sup stable finite-element pair best known as the MINI element, see \cref{subsec:DiscSpaces}. To the best of our knowledge, this contribution is original. 
The advantage of the hybrid CVFE-FEM approach is the fact that it can easily be applied to higher-order polynomial finite-element function spaces (e.g.\ $\PSpace_k,\;k>1$). 

\subsection{Choice of stable discrete function spaces}
\label{subsec:DiscSpaces}
For finite-element methods, it is well known, see e.g.~\cite{girault2012finite}, that the pairings $\PSpace^d_1 \times \PSpace_1$ or $\PSpace^d_1\times\PSpace_0$
for velocity and pressure, do not yield inf-sup-stable discretizations of \cref{eq:stokes_weak} (i.e.\ do not fulfill the LBB conditions~\cite{BrezziFortin1991}). As discussed in \cite{ye2001relationship}, it is therefore unlikely that a finite-volume scheme that uses the corresponding basis functions for interpolation yields a stable discretization scheme. 
Here, we construct a new control-volume finite-element scheme based on the basis functions of an inf-sup stable finite-element space pairing. Such a stable pairing is given by the MINI element~\cite{Arnold1984}, which is composed of $\PSpace_1$ finite-element functions enriched with a bubble function for the velocity ($\bubblebox{d}$) paired with the standard $\PSpace_1$ finite-element functions for the pressure.  

\paragraph{Velocity function space}
The degrees of freedom for the velocity are located at the vertices and the element center, as shown in \cref{fig:control-volumes}. Using the introduced notation, the set of positions of the degrees of freedom are given by $\calDofs = \calV \cup \lbrace \mathbf{x}_E \rbrace_{E \in \calM}$. The basis function corresponding to the element degree of freedom is constructed from the $\PSpace_1$ basis associated with element vertices, $\varphi_v$, and given as
\begin{equation}
\varphi_E :=  \left[ \prod_{v \in \calV_E} \varphi_v(\mathbf{x}_E)  \right]^{-1} \prod_{v \in \calV_E} \varphi_v.
\end{equation}
The function $\varphi_E$ is called a \emph{bubble function}, since it vanishes on the element boundary ($\varphi_E\arrowvert_{\partial E} = 0$) and is positive in the element interior. Moreover, $\varphi_E(\mathbf{x}_E) = 1$, as we normalized by the product of the $\PSpace_1$ basis functions.

Combining the $\PSpace_1$ basis functions with $\varphi_E$ does not yield an element-local basis that fulfills \cref{eq:localBasisUnity}. To achieve this, we propose to use instead the following modified $\PSpace_1$ basis functions,
\begin{equation}
\tilde{\varphi}_v := \varphi_v - \varphi_v(\mathbf{x}_E) \varphi_E, \quad \forall  v \in \calV_E.
\end{equation}

The local basis for the function space $\bubblespace$, is then defined by the basis functions $\lbrace \varphi_E \rbrace \cup \lbrace \tilde{\varphi}_v \; \vert \; v \in \calV_E \rbrace$. 
These basis functions are visualized on a reference element in \cref{fig:basisfunctions}. The definition of the respective basis function is given in the caption of each sub-figure. For this reference element it holds that $\mathbf{x}_E = ( \frac{1}{3}, \frac{1}{3})$ such that $\varphi_v(\mathbf{x}_E) = \frac{1}{3} $ from which the function definitions, provided in the figure captions, directly follows. \cref{eq:localBasisUnity} obviously holds and the basis functions sum up to one. 
\begin{figure}
    \centering
     \begin{subfigure}[b]{0.24\textwidth}
         \centering
         \includegraphics[width=\textwidth]{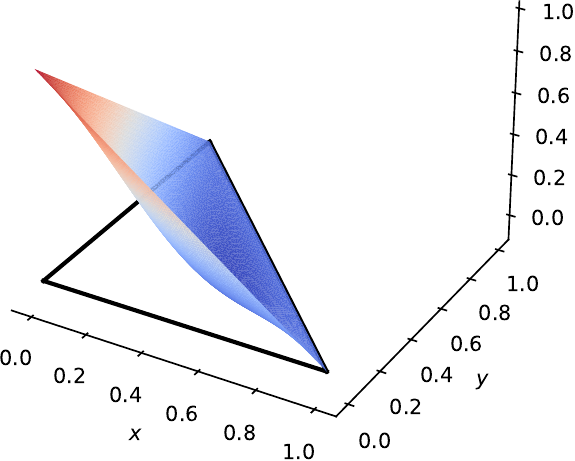}
         \caption{$\tilde{\varphi}_1 = 1-x -y - \frac{1}{3}\varphi_E$ }
         \label{fig:basis1}
     \end{subfigure}
     \hfill
     \begin{subfigure}[b]{0.24\textwidth}
         \centering
         \includegraphics[width=\textwidth]{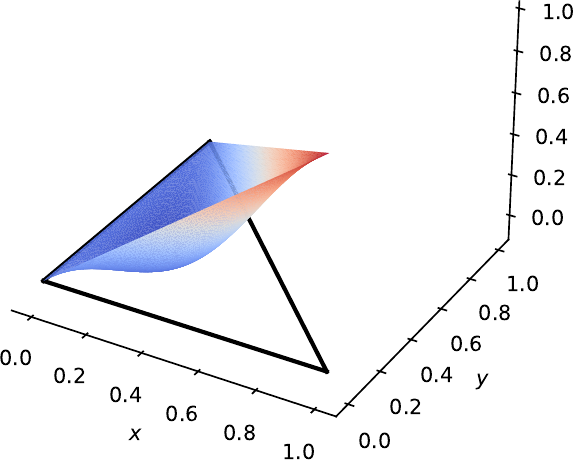}
         \caption{$\tilde{\varphi}_2 = x - \frac{1}{3}\varphi_E$}
         \label{fig:basis2}
     \end{subfigure}
     \hfill     
     \begin{subfigure}[b]{0.24\textwidth}
         \centering
         \includegraphics[width=\textwidth]{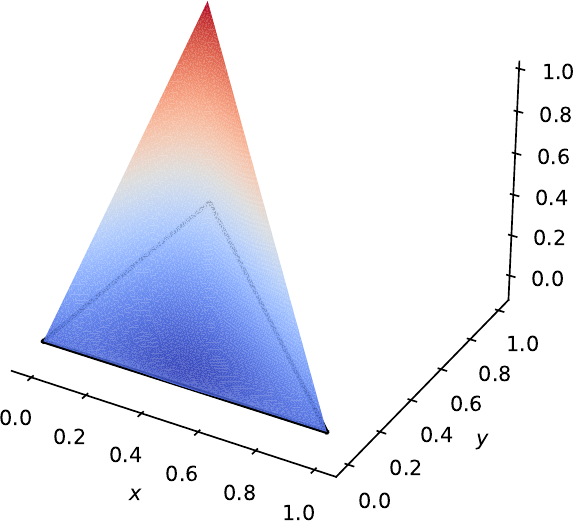}
         \caption{$\tilde{\varphi}_3 = y - \frac{1}{3}\varphi_E$}
         \label{fig:basis3}
     \end{subfigure}
     \hfill
     \begin{subfigure}[b]{0.24\textwidth}
         \centering
         \includegraphics[width=\textwidth]{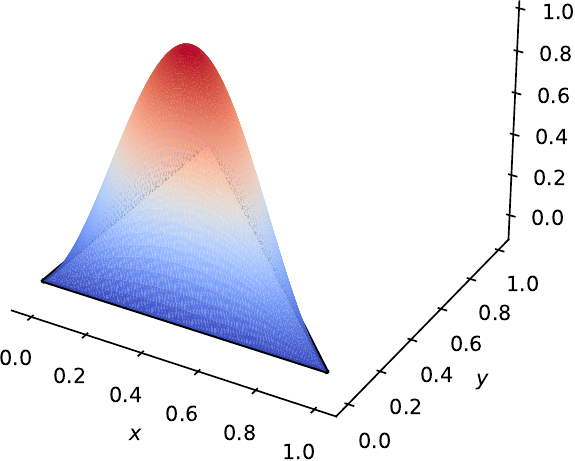}
         \caption{$\varphi_E = 27xy(1-x-y)$}
         \label{fig:basis4}
     \end{subfigure}
    \caption{\textbf{$\bubblespace$ basis functions on the two-dimensional reference element.} The function $\varphi_E$ is a bubble function (vanishes on all faces) on the shown element, and $\varphi_k$, $k=1,2,3$, are modified nodal basis functions.}
    \label{fig:basisfunctions}
\end{figure}

\paragraph{Pressure function space}
The pressure space for the MINI element is
given by the usual $\PSpace_1$ finite-element space. Degrees of freedom are located at the element vertices, see \cref{fig:control-volumes}.

\subsection{Construction of control volumes}
\label{subsec:cvs}
According to the chosen basis functions and the associated degrees of freedom, control volumes are constructed. In general, the element-local basis functions are smooth within each element but not $C^1$ across element boundaries (they may even be discontinuous depending on the chosen function space). Therefore, for being able to calculate gradients, needed for calculating fluxes, sub-control-volume faces are assumed to have intersections of zero measure with $\partial E$, i.e.\ $\meas{\sigma \cap \partial E} = 0$. This needs to be accounted for when constructing the control volumes.

In the following, we discuss the control volume construction for the function space bases of the previous section. We distinguish between non-overlapping CVFE, overlapping CVFE (cf. \cref{subsec:cvfeschemes}), and hybrid CVFE-FEM discretizations (cf. \cref{subsec:hcvfeschemes}). Here, these different approaches are applied only to the velocity space. For the pressure space, partitioning the domain into control volumes around each vertex is straightforward. The resulting non-overlapping CVFE discretization scheme for the pressure subspace exactly is the Box method ($\PSpace_1$-CVFE) as described in \cref{example:box}.

\paragraph{Non-overlapping CVFE method}
An example of non-overlapping control volumes for the two-dimensional case on a triangular element is shown in \cref{fig:control-volumes}. Here, we need to construct volumes for velocity degrees of freedom differently than for pressure degrees of freedom, meaning that flux balances for momentum and mass are given for different volumes. For the three-dimensional case, it is quite challenging to construct a non-overlapping partition, in particular, due to the requirement of zero-measure intersections $\meas{\sigma \cap \partial E} = 0$. While preparing this manuscript, we became aware of a very recently published work \cite{yang2023} that discusses the construction of a non-overlapping CVFE scheme for the MINI element, restricted to the case of $d=2$.  

\paragraph{Overlapping CVFE method}
As shown in \cref{fig:control-volumes}, the control volumes are chosen as the ``boxes'' (dual-mesh elements as in \cref{example:box}) constructed around each vertex and the (overlapping) element-centered control volume whose corner vertices coincide with the element's face centers $\mathbf{x}_\gridFace$.
With this choice, the control volumes of $\calT^{\vel}_\Omega$ are chosen the same as for the pressure degrees of freedom, meaning that for all $K \in \calT^{\vel}_\Omega$ both momentum and mass balances hold and both mass and momentum are therefore locally conserved by construction.
There are various ways to construct the additional volume around the bubble degree of freedom, thus giving some flexibility. The one presented in this work easily extends to three dimensions, i.e.\ tetrahedrons, yielding an octahedron. (In future work, it might be interesting to investigate the optimal choice of this overlapping control volume.)

\paragraph{Hybrid CVFE-FEM method}
For vertex degrees of freedom, the same control volumes as for the overlapping scheme are used, that is, ``boxes'' (dual-mesh elements as in \cref{example:box}) are constructed around each vertex. Since the bubble degrees of freedom are treated in a finite-element fashion, no additional control volumes need to be constructed. 

The function space splitting, as discussed in \cref{subsec:hcvfeschemes}, for the MINI element is given as follows
\begin{equation}
    \discspace^\downarrow = \text{span}\lbrace \varphi_v \rbrace_{v \in \calV}, \quad \discspace^\uparrow = \text{span}\lbrace \varphi_E \rbrace_{E \in \calM}.
\end{equation}
With this, \cref{eq:DiscMomentumBalanceFemSplitting} is enforced for the basis functions $\lbrace \varphi_E \rbrace_{E \in \calM}$. Since $\text{supp}(\varphi_E) = E$, the stencils for assembly of local residuals are the same as for the non-overlapping and overlapping schemes. The control volumes and the support region related to the bubble degrees of freedom are depicted in \cref{fig:control-volumes}. The advantage of this approach is the fact that it can easily be extended to higher-order polynomial spaces.

\subsection{Local mass and momentum conservation}
\label{sec:conservation}

The equivalence of the MINI element and equal-order stabilized Petrov-Galerkin finite-element schemes~\cite{Brezzi1984,Hughes1986} has been established by \cite{Pierre1988,Bank1990}. However, as evident by this equivalence, stability comes at the expense of degrading local mass conservation~\cite{Cioncolini2019}. Similarly, the stabilized lowest-order $\PSpace_1-\PSpace_0$-based CVFE scheme of \cite{Zhang2015} also sacrifices local mass conservation for stability. Local conservation properties may be achieved through post-processing \cite{Matthies2007}.

All schemes presented in this work satisfy local mass conservation for the control volumes $\calT^{\p}_{\Omega}$ and local momentum conservation for the control volumes $\calT^{\vel}_{\Omega}$. This is directly given by the fact that \cref{eq:DiscEleMassBalance,eq:DiscEleMomentumBalance} represent local flux balances and the property that there is a subset of control volumes that form a partition such that for each $\sigma \in \calE_K$, $K\in \calT_\Omega$, there is another control volume $L\in \calT_\Omega$ such that $\sigma \in \calE_L$ and $\n_{L,\sigma} = - \n_{K,\sigma}$. Finally, using the continuity of $\vel_h$ on all $\sigma \in \calE^{\p}_{\Omega}$ yields local mass conservation, whereas the continuity of $p_h$ and $\nabla \vel_h$ on $\sigma \in \calE^{\vel}_{\Omega}$ yields local momentum conservation. Therefore, all schemes are locally mass- and momentum conservative but they differ in the control volumes where these hold. 

\begin{remark}
If the integrals in \cref{eq:DiscEleMassBalance,eq:DiscEleMomentumBalance} are approximated by evaluating the integrands at the face centers, then the above continuity assumptions can be weakened and only continuity at the sub-control-volume face centers is required to guarantee local conservation. All presented schemes satisfy continuity on the whole sub-control-volume faces. Therefore, appropriate quadrature rules are used for calculating integrals. 
\end{remark}

\section{Numerical results}
\label{sec:numeric}

Using standard finite-element theory, it can be proven (e.g.~\cite{Boffi2013}) that the MINI element is linearly convergent in both velocity ($H^1$-norm) and pressure ($L^2$-norm). Recently, superconvergence $\mathcal{O}(h^{3/2})$ in the pressure and in the linear part of the velocity has been shown by \cite{Eichel2011} for certain types of structured meshes. Recent numerical evidence suggests, that superconvergence can also be observed on unstructured grids when using mesh optimization tools (for Delaunay meshes) \cite{Cioncolini2019}.

For the finite-element method using the MINI element, as pointed out in \cite{Verfrth1989} and further discussed by \citep{Bank1990,Russo1995,Kim2000}, the linear part of the discrete velocity seems to be a better approximation of the solution $\vel$ than the full discrete velocity including the bubble function contribution, and some a-posterior error estimates are therefore based on the linear velocity part. In particular, this suggests that the bubble contribution may only be needed for stability but does not improve the approximation \cite{Cioncolini2019}. In the following, we reiterate on these observations in the light of the presented control-volume finite-element schemes. 

Here, we consider four different schemes constructed based on the MINI element:
\begin{itemize}
    \item[] $\bubbleboxname{d}{nov}$: non-overlapping CVFE scheme
    \item[] $\bubbleboxname{d}{ov}$: overlapping CVFE scheme
    \item[] $\bubbleboxname{d}{hy}$: hybrid CVFE-FEM scheme
    \item[] $\bubbleboxname{d}{fem}$: classical MINI FEM
\end{itemize}
As mentioned above, the non-overlapping scheme cannot be easily extended to 3D, which is why for such cases, only the results of the other schemes are shown. 

All presented results have been conducted with the open-source simulator \dumux~\cite{Koch2021Dumux} which is based on DUNE~\cite{Dune2021}. We included an implementation of the $\bubbleboxname{d}{ov}$ discretization scheme as part of \dumux and made it publicly available as part of \dumux release 3.7~\cite{Oukili2023}.

\subsection{Preconditioner and linear solver iterations}
Discretizing the continuous problem results in a linear system $Ax = b$ which is often solved with iterative methods such as Krylov subspace methods. The operator preconditioning framework~\cite{Hiptmair2006,Mardal2010} provides a recipe to design parameter-robust preconditioners for Krylov subspace methods.

For the Stokes problem, \cref{eq:stokes}, it is known that the coefficient operator
\begin{equation}
    \mathcal{A} = \begin{bmatrix}
    -2\mu \boldsymbol{D} & \nabla \\
     \operatorname{div} & \\
    \end{bmatrix}
\end{equation}
is an isomorphism, cf.~\cite[Def.~3.5]{Braess2007}, mapping to the dual space $\boldsymbol{H}^{-1} \times L^2$ and a canonical choice for a preconditioner of the continuous problem is the Riesz map~\cite{Mardal2010},
\begin{equation}
    \mathcal{B} = \begin{bmatrix}
    -2\mu \boldsymbol{D} & \\
    & (2\mu)^{-1} I
    \end{bmatrix}^{-1},
\end{equation}
which is an isomorphism that maps from the dual space back to $\boldsymbol{H}^1_0 \times L^2$. It follows that $\mathcal{B}\mathcal{A}$ is an endomorphism (mapping from $\boldsymbol{H}^1_0 \times L^2$ to $\boldsymbol{H}^1_0 \times L^2$).

A necessary condition for the mappings to be isomorphisms is the inf-sup condition (e.g.~\cite[Thm.~3.6]{Braess2007}). It is therefore argued within the operator preconditioning framework~\cite{Mardal2010} that only if \cref{eq:stokes_weak} is discretized by an inf-sup stable discretization scheme (cf.~\cite[Lemma~3.7 \& Remark]{Braess2007}), the discrete version of $\mathcal{B}$ is a robust preconditioner in the sense that the condition number of the preconditioned system (and therefore the number of iterations of the preconditioned Krylov subspace method) is uniformly bounded independently of the discretization parameter (grid refinement) or choice of model parameters. In this work, we will use this argument to provide a strong indication of (inf-sup) stability of the suggested discretization methods by investigating numerically the stability of iteration counts of a left-preconditioned GMRes solver with grid refinement.

To this end, we are solving the discrete linear system of equations $P(J\Delta x - r) = 0$ for the update vector $\Delta x = \left[ \Delta \vel_h, \Delta p_h \right ]^T$ (update with respect to the initial guess $x_0$), where
\begin{equation}
    J =
    \begin{bmatrix} A & B \\ C & 0 \end{bmatrix},
    \quad P =
    \begin{bmatrix} A^{-1} & 0 \\ 0 & \tilde{S}^{-1} \end{bmatrix}
    \begin{bmatrix} 1 & 0 \\ -CA^{-1} & 1 \end{bmatrix},
    \quad \tilde{S} := (2 \mu)^{-1} I,
\end{equation}
denote the discrete Jacobian operator, the discrete preconditioner operator---an approximate inverse of $J$ with block-triagonal structure, and the approximate Schur complement operator $\tilde{S}$, respectively. Note that $\tilde{S}$ corresponds to a viscosity-scaled pressure mass matrix.

The matrices $A$ and $\tilde{S}$ are inverted using LU decomposition. To ensure the invertibility of $A$, we choose mixed boundary conditions such that $\partial \Omega_D \neq \emptyset$ in all numerical tests.
The iterative solver starts from a random initial vector which represents a random function in the chosen discrete function space. The degrees of freedom on the Dirichlet boundary are zero and the remaining degrees of freedom are drawn from the uniform distribution on $[-1, 1]$. The solver terminates once the preconditioned residual norm is reduced by a factor $10^{10}$.

\subsection{Convergence tests}
Within this section, four different convergence test cases are investigated: (1) 2D Donea-Huerta \cite{Donea2003}; (2) 2D Bercovier–Engelman \cite{Boyer2017}; (3) 3D Taylor-Green \cite{Boyer2017}; (4) Stokes flow in a circular tube (radially-symmetric solution in cylinder coordinates).

The convergence rates are investigated for the meshes shown in \cref{fig:grids}, i.e.\ Delaunay and randomly distorted grids. Delaunay meshes are generated with gmsh \cite{gmsh}, and the randomly distorted meshes are constructed by shifting the grid vertices of a structured grid by randomly generated scaling factors (maximum scaling is set to 20\% of the discretization length).
We report $L^2$ and $H^1$ convergence rates and the number of linear solver iterations (it). The discretization lengths, $h^\p$ and $h^{\vel}$, are calculated as $(|\Omega| |\calDofs^{\p}|^{-1})^{\frac{1}{d}}$, $(|\Omega| |\calDofs^{\vel}|^{-1})^{\frac{1}{d}}$, where $|\calDofs^{\p}|,|\calDofs^{\vel}|$ correspond to the number of degrees of freedom. 
The source terms, calculated based on the analytical solutions, need to be integrated for the presented schemes, which is done by using a quadrature rule.

\begin{figure}
    \centering
     \begin{subfigure}[b]{0.24\textwidth}
         \centering
         \includegraphics[width=\textwidth]{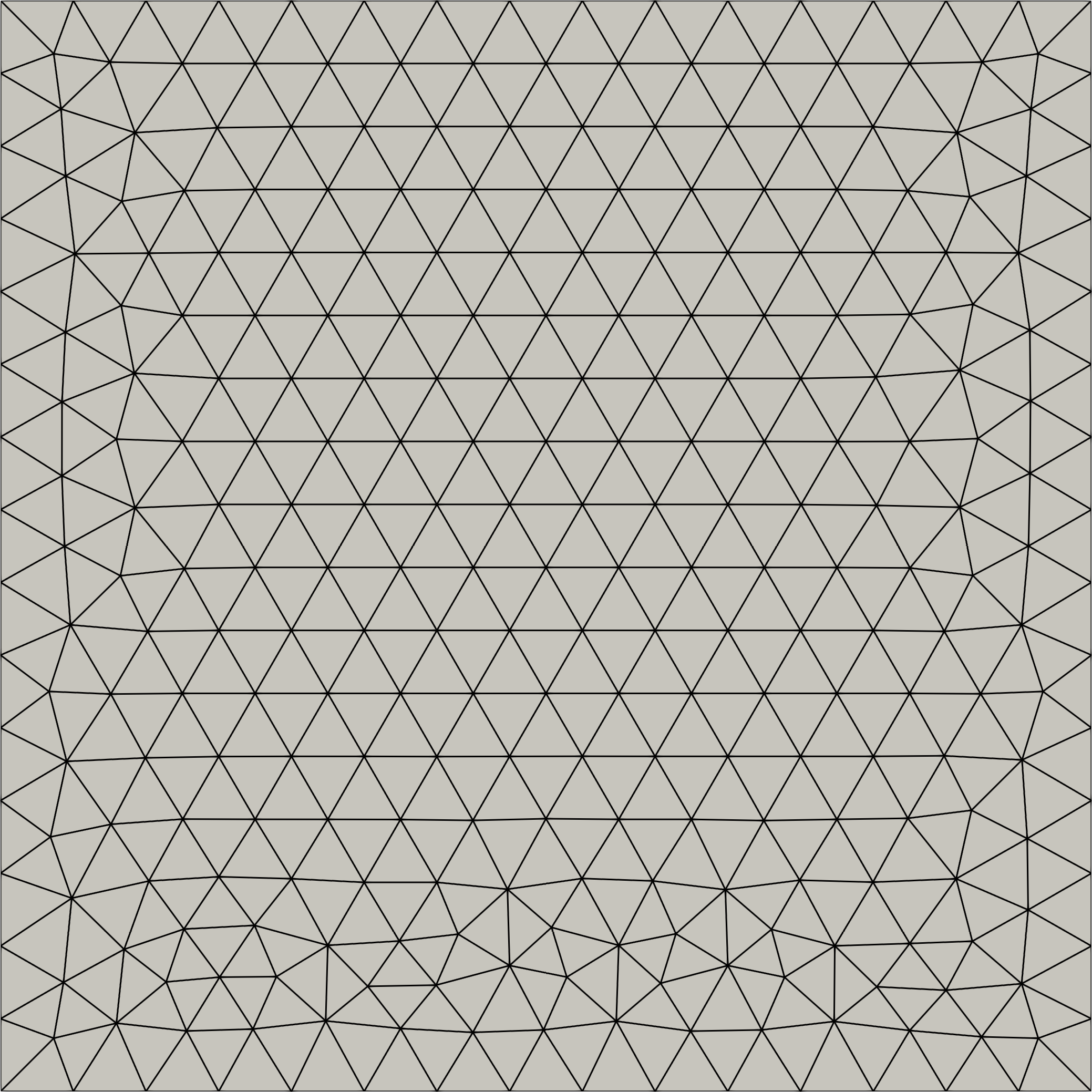}
         \caption{2D Delaunay}
         \label{fig:delaunaygrid2d}
     \end{subfigure}
     \hfill
     \begin{subfigure}[b]{0.24\textwidth}
         \centering
         \includegraphics[width=\textwidth]{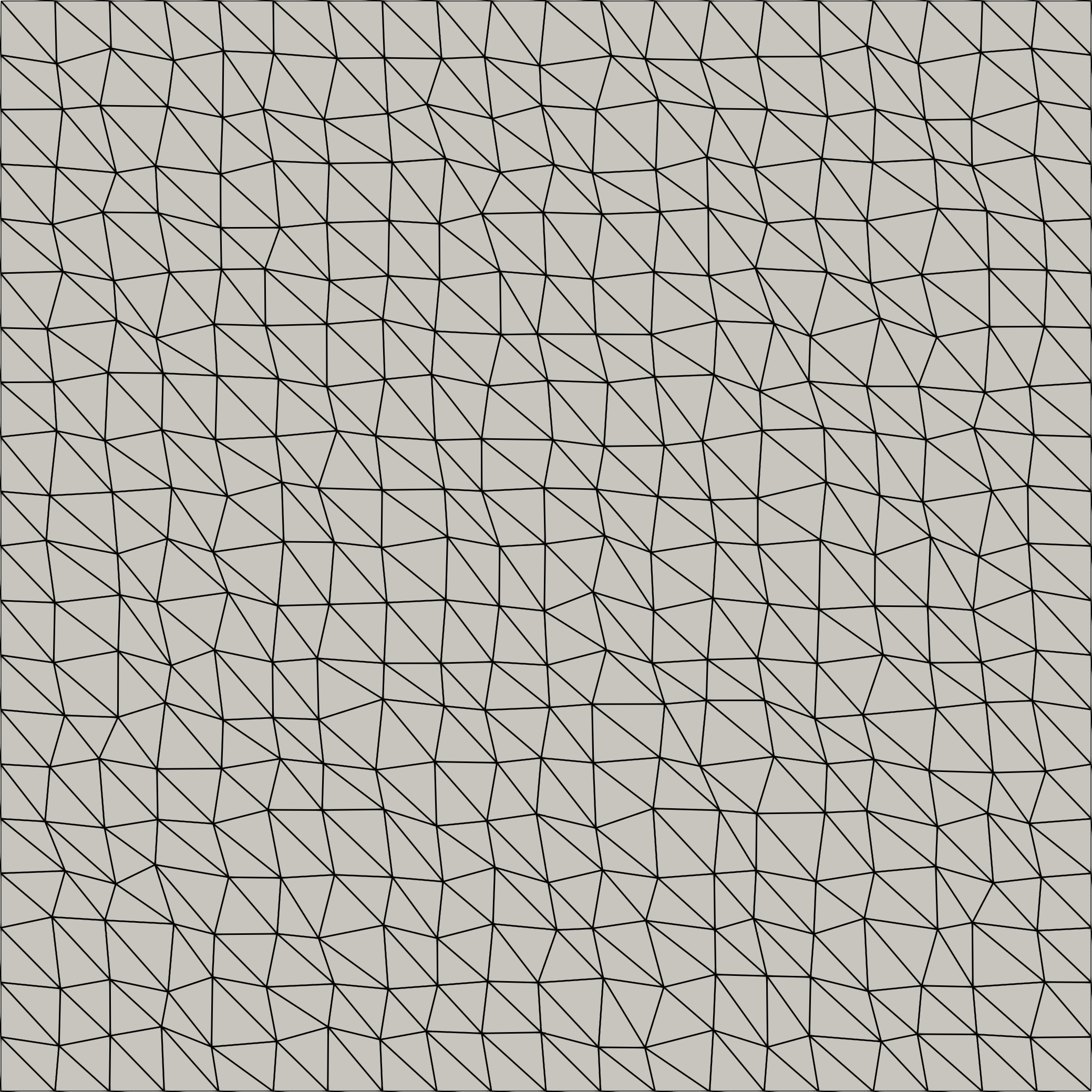}
         \caption{2D randomly distorted}
         \label{fig:randomgrid2d}
     \end{subfigure}
     \hfill     
     \begin{subfigure}[b]{0.24\textwidth}
         \centering
         \includegraphics[width=\textwidth]{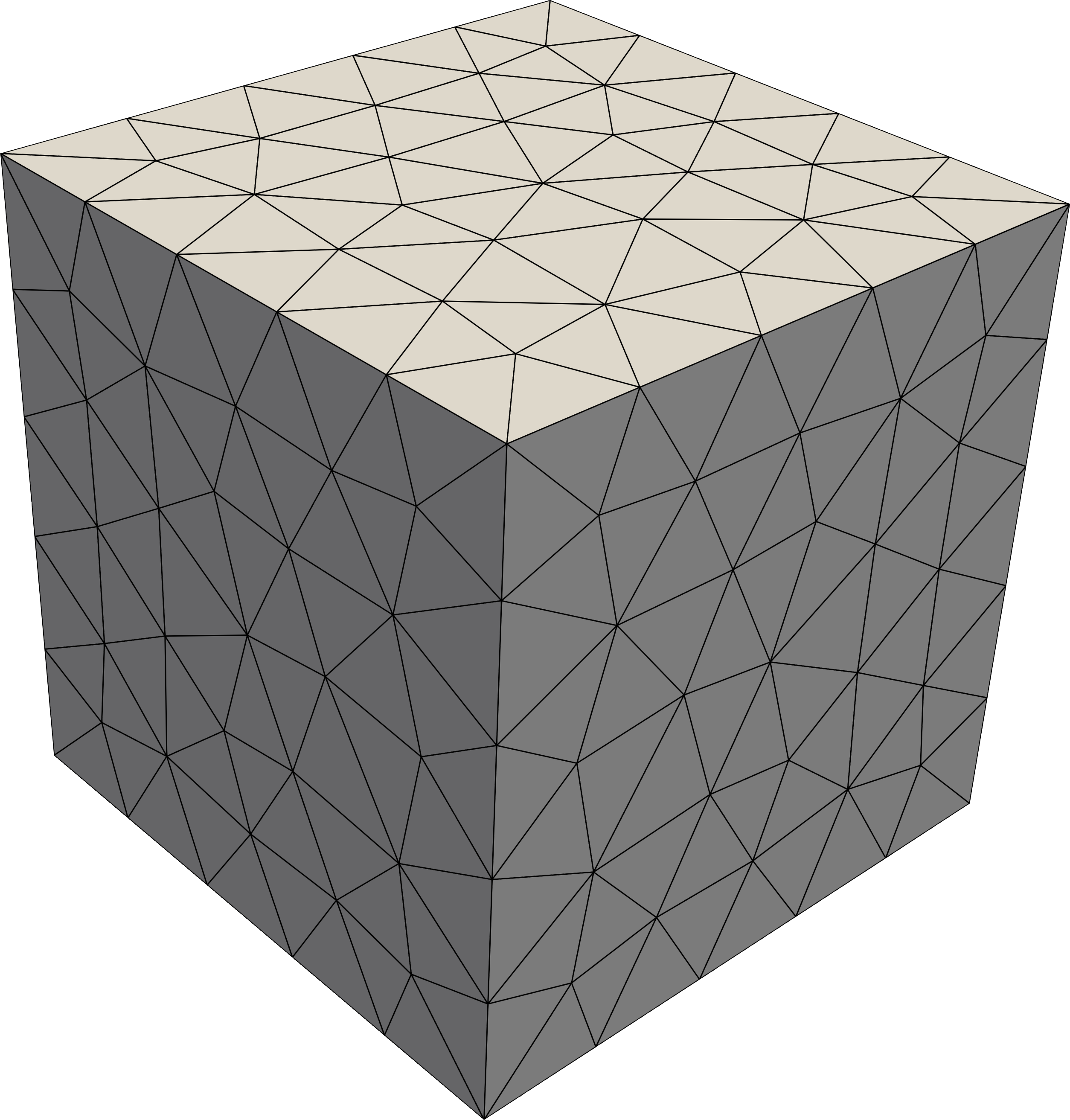}
         \caption{3D Delaunay}
         \label{fig:delaunaygrid3d}
     \end{subfigure}
     \hfill
     \begin{subfigure}[b]{0.24\textwidth}
         \centering
         \includegraphics[width=\textwidth]{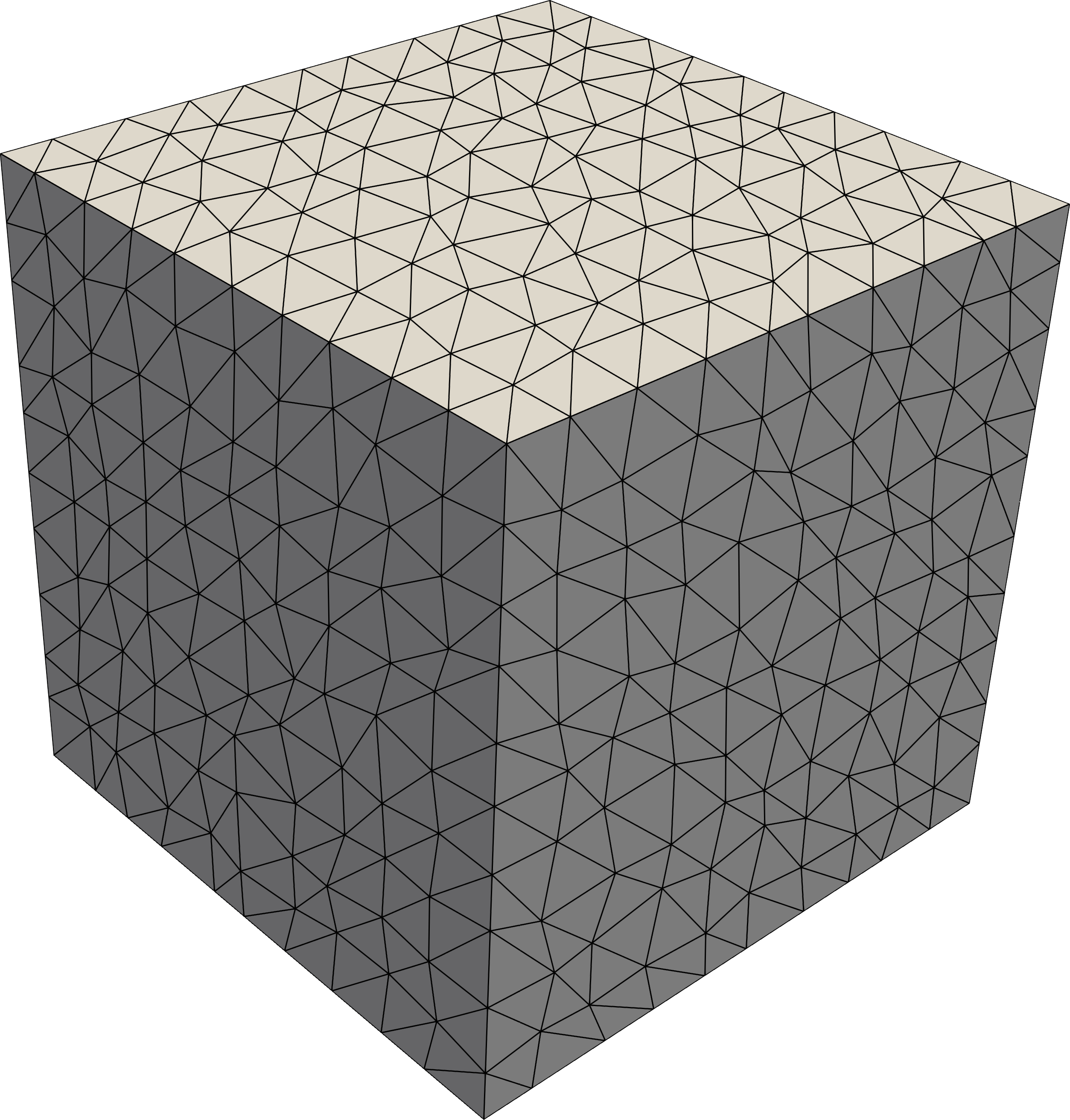}
         \caption{3D randomly distorted}
         \label{fig:randomgrid3d}
     \end{subfigure}
    \caption{Grids}
    \label{fig:grids}
\end{figure}

\subsubsection{Donea-Huerta}
This test case has been presented in \cite{Donea2003}, with an analytical solution as given in \labelcref{sec:appendix:donea}. Mixed boundary conditions matching the analytical solution are set for the momentum balance, i.e.\ Dirichlet conditions on the left and lower boundaries and Neumann conditions on the right and upper boundaries of the domain $\Omega = (0,1)^2$.
\begin{table}[!hbtp]
\begin{small}
  \centering
  \caption{\textbf{Convergence study}. Donea-Huerta test case (2D) on a Delaunay grid.}
  \label{tab:donea-delaunay}
  \begin{tabular}{l*{9}{l}}
    \toprule
scheme &     $h^p$ & $\lVert p_h - p\rVert_{L^2}$ &rate & $h^{\boldsymbol{v}}$ & $\lVert \boldsymbol{v}_h - \boldsymbol{v}\rVert_{L^2}$ &rate & $\lVert \boldsymbol{v}_h - \boldsymbol{v}\rVert_{H^1}$ &rate & it \\
    \midrule
    \multirow{6}{*}{\bubbleboxovtable{2}} & 1.0e-01 & 5.64e-03 & - & 6.2e-02 & 7.48e-04 & - & 1.38e-02 & - & 24 \\
     & 5.7e-02 & 2.51e-03 & 1.43 & 3.4e-02 & 2.09e-04 & 2.16 & 7.44e-03 & 1.05 & 26 \\
     & 3.1e-02 & 8.72e-04 & 1.70 & 1.8e-02 & 5.75e-05 & 2.02 & 3.77e-03 & 1.06 & 27 \\
     & 1.6e-02 & 3.20e-04 & 1.53 & 9.3e-03 & 1.50e-05 & 2.03 & 1.93e-03 & 1.01 & 27 \\
     & 8.1e-03 & 1.18e-04 & 1.45 & 4.7e-03 & 3.87e-06 & 1.96 & 9.67e-04 & 1.01 & 27 \\
     & 4.1e-03 & 4.30e-05 & 1.48 & 2.4e-03 & 9.76e-07 & 2.01 & 4.88e-04 & 1.00 & 27 \\
    \hline
    \multirow{6}{*}{\bubbleboxhytable{2}} & 1.0e-01 & 6.93e-03 & - & 6.2e-02 & 7.19e-04 & - & 1.39e-02 & - & 26 \\
     & 5.7e-02 & 3.19e-03 & 1.37 & 3.4e-02 & 2.04e-04 & 2.13 & 7.52e-03 & 1.04 & 28 \\
     & 3.1e-02 & 1.12e-03 & 1.68 & 1.8e-02 & 5.68e-05 & 2.00 & 3.81e-03 & 1.06 & 29 \\
     & 1.6e-02 & 4.11e-04 & 1.54 & 9.3e-03 & 1.49e-05 & 2.01 & 1.96e-03 & 1.00 & 29 \\
     & 8.1e-03 & 1.53e-04 & 1.44 & 4.7e-03 & 3.86e-06 & 1.96 & 9.78e-04 & 1.01 & 29 \\
     & 4.1e-03 & 5.54e-05 & 1.49 & 2.4e-03 & 9.76e-07 & 2.01 & 4.94e-04 & 1.00 & 29 \\
    \hline
    \multirow{6}{*}{\bubbleboxnovtable{2}} & 1.0e-01 & 5.02e-03 & - & 6.2e-02 & 4.87e-04 & - & 1.36e-02 & - & 24 \\
     & 5.7e-02 & 2.39e-03 & 1.31 & 3.4e-02 & 1.40e-04 & 2.11 & 7.41e-03 & 1.02 & 26 \\
     & 3.1e-02 & 8.53e-04 & 1.66 & 1.8e-02 & 3.87e-05 & 2.01 & 3.76e-03 & 1.06 & 27 \\
     & 1.6e-02 & 3.16e-04 & 1.52 & 9.3e-03 & 1.03e-05 & 2.00 & 1.93e-03 & 1.00 & 27 \\
     & 8.1e-03 & 1.18e-04 & 1.45 & 4.7e-03 & 2.66e-06 & 1.96 & 9.67e-04 & 1.01 & 27 \\
     & 4.1e-03 & 4.27e-05 & 1.48 & 2.4e-03 & 6.75e-07 & 2.00 & 4.88e-04 & 1.00 & 27 \\
    \hline
    \multirow{6}{*}{\bubbleboxfemtable{2}} & 1.0e-01 & 8.72e-03 & - & 6.2e-02 & 7.67e-04 & - & 1.43e-02 & - & 25 \\
     & 5.7e-02 & 4.63e-03 & 1.12 & 3.4e-02 & 2.11e-04 & 2.19 & 7.67e-03 & 1.05 & 28 \\
     & 3.1e-02 & 1.72e-03 & 1.59 & 1.8e-02 & 4.97e-05 & 2.26 & 3.84e-03 & 1.08 & 29 \\
     & 1.6e-02 & 6.30e-04 & 1.53 & 9.3e-03 & 1.22e-05 & 2.12 & 1.97e-03 & 1.01 & 29 \\
     & 8.1e-03 & 2.39e-04 & 1.42 & 4.7e-03 & 2.90e-06 & 2.08 & 9.80e-04 & 1.01 & 29 \\
     & 4.1e-03 & 8.59e-05 & 1.50 & 2.4e-03 & 7.30e-07 & 2.01 & 4.94e-04 & 1.00 & 29 \\
    \hline
  \bottomrule
  \end{tabular}
  \end{small}
\end{table}

\begin{table}[!hbtp]
\begin{small}
  \centering
  \caption{\textbf{Convergence study}. Donea-Huerta test case (2D) on a randomly distorted grid.}
  \label{tab:donea-distorted}
  \begin{tabular}{l*{9}{l}}
    \toprule
scheme &     $h^p$ & $\lVert p_h - p\rVert_{L^2}$ &rate & $h^{\boldsymbol{v}}$ & $\lVert \boldsymbol{v}_h - \boldsymbol{v}\rVert_{L^2}$ &rate & $\lVert \boldsymbol{v}_h - \boldsymbol{v}\rVert_{H^1}$ &rate & it \\
    \midrule
    \multirow{6}{*}{\bubbleboxovtable{2}} & 9.1e-02 & 7.75e-03 & - & 5.6e-02 & 6.54e-04 & - & 1.55e-02 & - & 27 \\
     & 4.8e-02 & 3.10e-03 & 1.41 & 2.8e-02 & 1.89e-04 & 1.84 & 7.96e-03 & 0.99 & 30 \\
     & 2.4e-02 & 1.40e-03 & 1.19 & 1.4e-02 & 5.05e-05 & 1.93 & 4.01e-03 & 1.00 & 30 \\
     & 1.2e-02 & 6.32e-04 & 1.17 & 7.2e-03 & 1.27e-05 & 2.01 & 2.01e-03 & 1.00 & 30 \\
     & 6.2e-03 & 2.99e-04 & 1.09 & 3.6e-03 & 3.21e-06 & 1.98 & 1.01e-03 & 1.00 & 30 \\
     & 3.1e-03 & 1.45e-04 & 1.05 & 1.8e-03 & 7.94e-07 & 2.02 & 5.02e-04 & 1.00 & 30 \\
    \hline
    \multirow{6}{*}{\bubbleboxhytable{2}} & 9.1e-02 & 9.17e-03 & - & 5.6e-02 & 6.74e-04 & - & 1.56e-02 & - & 30 \\
     & 4.8e-02 & 3.67e-03 & 1.42 & 2.8e-02 & 1.94e-04 & 1.84 & 8.01e-03 & 0.99 & 33 \\
     & 2.4e-02 & 1.70e-03 & 1.15 & 1.4e-02 & 5.18e-05 & 1.93 & 4.03e-03 & 1.00 & 34 \\
     & 1.2e-02 & 7.58e-04 & 1.19 & 7.2e-03 & 1.30e-05 & 2.01 & 2.02e-03 & 1.00 & 34 \\
     & 6.2e-03 & 3.56e-04 & 1.10 & 3.6e-03 & 3.29e-06 & 1.99 & 1.01e-03 & 1.00 & 34 \\
     & 3.1e-03 & 1.72e-04 & 1.06 & 1.8e-03 & 8.15e-07 & 2.02 & 5.05e-04 & 1.00 & 34 \\
    \hline
    \multirow{6}{*}{\bubbleboxnovtable{2}} & 9.1e-02 & 7.42e-03 & - & 5.6e-02 & 6.92e-04 & - & 1.55e-02 & - & 27 \\
     & 4.8e-02 & 3.07e-03 & 1.36 & 2.8e-02 & 2.00e-04 & 1.83 & 7.97e-03 & 0.98 & 30 \\
     & 2.4e-02 & 1.40e-03 & 1.17 & 1.4e-02 & 5.27e-05 & 1.95 & 4.01e-03 & 1.00 & 30 \\
     & 1.2e-02 & 6.33e-04 & 1.17 & 7.2e-03 & 1.33e-05 & 2.00 & 2.01e-03 & 1.00 & 30 \\
     & 6.2e-03 & 3.00e-04 & 1.09 & 3.6e-03 & 3.36e-06 & 1.99 & 1.01e-03 & 1.00 & 30 \\
     & 3.1e-03 & 1.45e-04 & 1.05 & 1.8e-03 & 8.32e-07 & 2.02 & 5.02e-04 & 1.00 & 30 \\
    \hline
    \multirow{6}{*}{\bubbleboxfemtable{2}} & 9.1e-02 & 1.21e-02 & - & 5.6e-02 & 1.02e-03 & - & 1.63e-02 & - & 31 \\
     & 4.8e-02 & 5.03e-03 & 1.36 & 2.8e-02 & 2.79e-04 & 1.92 & 8.19e-03 & 1.02 & 34 \\
     & 2.4e-02 & 2.49e-03 & 1.05 & 1.4e-02 & 7.12e-05 & 2.00 & 4.09e-03 & 1.01 & 35 \\
     & 1.2e-02 & 1.10e-03 & 1.19 & 7.2e-03 & 1.78e-05 & 2.01 & 2.04e-03 & 1.01 & 35 \\
     & 6.2e-03 & 5.12e-04 & 1.12 & 3.6e-03 & 4.44e-06 & 2.01 & 1.02e-03 & 1.00 & 35 \\
     & 3.1e-03 & 2.47e-04 & 1.06 & 1.8e-03 & 1.11e-06 & 2.01 & 5.08e-04 & 1.00 & 35 \\
    \hline
  \bottomrule
  \end{tabular}
  \end{small}
\end{table}

The numerical results are shown in \cref{tab:donea-delaunay,tab:donea-distorted}. For all schemes, a 2nd order $L^2$ and a first order $H^1$ convergence rate is observed for the velocity on both grid types. All schemes also show $L^2$ super-convergence in pressure, i.e.\ $\mathcal{O}(h^{\frac{3}{2}})$, on Delaunay grids, whereas first-order convergence is observed on randomly distorted grids. 
$H^1$ velocity errors are almost the same for all schemes. $L^2$ velocity errors are the smallest for the non-overlapping scheme on the Delaunay grids but not on the randomly distorted grids. There, the hybrid and overlapping schemes have the smallest $L^2$ velocity errors. This is most likely caused by the fact that the non-overlapping control volumes are more sensitive to grid distortion. 
The smallest pressure errors are observed for the overlapping and the non-overlapping schemes.

The number of linear solver iterations (it) is stable with grid refinement, indicating uniform boundedness of the condition number of the discrete preconditioned linear system in the discretization parameter $h$. Since this property valid for the continuous system is only expected to carry over to the discrete system if the discretization is inf-sup stable, bounded condition numbers indicate inf-sup stability of the discretization scheme.
The number of iterations appears to be slightly larger for the FEM than for the CVFE schemes, except for the hybrid scheme. This is consistently observed in all subsequent tests as well.

\subsubsection{Bercovier–Engelman}
On $\Omega = \left(0, 1\right)^2$ and with $\mu = 1$, an exact solution for homogeneous Dirichlet boundary conditions is given in \cite{Boyer2017} for $\vel = (v_x, v_y)^T$:
\begin{align}
    v_x(x,y) = h(x, y), \quad
    v_y(x,y) = -h(y, x), \quad
    p(x,y) = \left(x - 0.5\right)\left(y - 0.5\right),
\end{align}
where $h(x,y) = -256x^2(x - 1)^2 y(y - 1)(2y - 1)$, and for the right-hand side
\begin{align}
    f_x(x,y) = g(x,y) + \left(y - 0.5\right), \quad
    f_y(x,y) = -g(y,x) + \left(x - 0.5\right),
\end{align}
where $g(x,y) = 256 \left[ x^2(x-1)^2(12y-6)+y(y-1)(2y-1)(12x^2 - 12x + 2) \right]$.

We remark that since the analytical pressure solution is bi-linear it cannot be exactly approximated by the $\boxspace$ basis functions for pressure. 
Mixed boundary conditions matching the analytical solution are set for the momentum balance, i.e.\ Dirichlet conditions on the left and lower boundaries and Neumann conditions on the right and upper boundaries.

\begin{table}[!hbtp]
\begin{small}
  \centering
  \caption{\textbf{Convergence study}. Bercovier–Engelman test case (2D) on a Delaunay grid.}
  \label{tab:engelman-delaunay}
  \begin{tabular}{l*{9}{l}}
    \toprule
scheme &     $h^p$ & $\lVert p_h - p\rVert_{L^2}$ &rate & $h^{\boldsymbol{v}}$ & $\lVert \boldsymbol{v}_h - \boldsymbol{v}\rVert_{L^2}$ &rate & $\lVert \boldsymbol{v}_h - \boldsymbol{v}\rVert_{H^1}$ &rate & it \\
    \midrule
    \multirow{6}{*}{\bubbleboxovtable{2}} & 1.0e-01 & 6.51e-01 & - & 6.2e-02 & 4.73e-02 & - & 1.74e+00 & - & 24 \\
     & 5.7e-02 & 3.05e-01 & 1.34 & 3.4e-02 & 1.28e-02 & 2.21 & 9.49e-01 & 1.03 & 26 \\
     & 3.1e-02 & 1.09e-01 & 1.66 & 1.8e-02 & 3.62e-03 & 1.98 & 4.82e-01 & 1.06 & 26 \\
     & 1.6e-02 & 4.04e-02 & 1.51 & 9.3e-03 & 9.37e-04 & 2.04 & 2.48e-01 & 1.00 & 26 \\
     & 8.1e-03 & 1.51e-02 & 1.44 & 4.7e-03 & 2.42e-04 & 1.96 & 1.24e-01 & 1.01 & 26 \\
     & 4.1e-03 & 5.48e-03 & 1.48 & 2.4e-03 & 6.16e-05 & 2.00 & 6.25e-02 & 1.00 & 27 \\
    \hline
    \multirow{6}{*}{\bubbleboxhytable{2}} & 1.0e-01 & 8.25e-01 & - & 6.2e-02 & 4.74e-02 & - & 1.76e+00 & - & 26 \\
     & 5.7e-02 & 3.95e-01 & 1.31 & 3.4e-02 & 1.31e-02 & 2.18 & 9.60e-01 & 1.03 & 28 \\
     & 3.1e-02 & 1.41e-01 & 1.65 & 1.8e-02 & 3.61e-03 & 2.01 & 4.88e-01 & 1.06 & 28 \\
     & 1.6e-02 & 5.22e-02 & 1.52 & 9.3e-03 & 9.40e-04 & 2.03 & 2.51e-01 & 1.00 & 29 \\
     & 8.1e-03 & 1.96e-02 & 1.43 & 4.7e-03 & 2.42e-04 & 1.97 & 1.25e-01 & 1.01 & 29 \\
     & 4.1e-03 & 7.08e-03 & 1.49 & 2.4e-03 & 6.17e-05 & 1.99 & 6.32e-02 & 1.00 & 29 \\
    \hline
    \multirow{6}{*}{\bubbleboxnovtable{2}} & 1.0e-01 & 6.00e-01 & - & 6.2e-02 & 4.83e-02 & - & 1.73e+00 & - & 24 \\
     & 5.7e-02 & 2.97e-01 & 1.25 & 3.4e-02 & 1.36e-02 & 2.15 & 9.48e-01 & 1.02 & 26 \\
     & 3.1e-02 & 1.07e-01 & 1.63 & 1.8e-02 & 3.45e-03 & 2.14 & 4.81e-01 & 1.06 & 26 \\
     & 1.6e-02 & 4.02e-02 & 1.50 & 9.3e-03 & 8.90e-04 & 2.05 & 2.48e-01 & 1.00 & 26 \\
     & 8.1e-03 & 1.50e-02 & 1.44 & 4.7e-03 & 2.21e-04 & 2.02 & 1.24e-01 & 1.01 & 26 \\
     & 4.1e-03 & 5.46e-03 & 1.48 & 2.4e-03 & 5.64e-05 & 1.99 & 6.25e-02 & 1.00 & 27 \\
    \hline
    \multirow{6}{*}{\bubbleboxfemtable{2}} & 1.0e-01 & 1.11e+00 & - & 6.2e-02 & 9.81e-02 & - & 1.82e+00 & - & 26 \\
     & 5.7e-02 & 5.92e-01 & 1.11 & 3.4e-02 & 2.70e-02 & 2.19 & 9.82e-01 & 1.05 & 28 \\
     & 3.1e-02 & 2.19e-01 & 1.59 & 1.8e-02 & 6.37e-03 & 2.26 & 4.92e-01 & 1.08 & 28 \\
     & 1.6e-02 & 8.07e-02 & 1.53 & 9.3e-03 & 1.56e-03 & 2.12 & 2.52e-01 & 1.01 & 28 \\
     & 8.1e-03 & 3.06e-02 & 1.42 & 4.7e-03 & 3.71e-04 & 2.08 & 1.25e-01 & 1.01 & 28 \\
     & 4.1e-03 & 1.10e-02 & 1.50 & 2.4e-03 & 9.34e-05 & 2.01 & 6.33e-02 & 1.00 & 28 \\
    \hline
  \bottomrule
  \end{tabular}
  \end{small}
\end{table}

\begin{table}[!hbtp]
\begin{small}
  \centering
  \caption{\textbf{Convergence study}. Bercovier–Engelman test case (2D) on randomly distorted grid.}
  \label{tab:engelman-distorted}
  \begin{tabular}{l*{9}{l}}
    \toprule
scheme &     $h^p$ & $\lVert p_h - p\rVert_{L^2}$ &rate & $h^{\boldsymbol{v}}$ & $\lVert \boldsymbol{v}_h - \boldsymbol{v}\rVert_{L^2}$ &rate & $\lVert \boldsymbol{v}_h - \boldsymbol{v}\rVert_{H^1}$ &rate & it \\
    \midrule
    \multirow{6}{*}{\bubbleboxovtable{2}} & 9.1e-02 & 9.89e-01 & - & 5.6e-02 & 6.62e-02 & - & 1.98e+00 & - & 28 \\
     & 4.8e-02 & 3.94e-01 & 1.42 & 2.8e-02 & 1.97e-02 & 1.79 & 1.02e+00 & 0.98 & 30 \\
     & 2.4e-02 & 1.79e-01 & 1.18 & 1.4e-02 & 5.15e-03 & 1.96 & 5.13e-01 & 1.00 & 29 \\
     & 1.2e-02 & 8.08e-02 & 1.17 & 7.2e-03 & 1.30e-03 & 2.00 & 2.57e-01 & 1.00 & 29 \\
     & 6.2e-03 & 3.83e-02 & 1.09 & 3.6e-03 & 3.27e-04 & 1.99 & 1.29e-01 & 1.00 & 30 \\
     & 3.1e-03 & 1.85e-02 & 1.05 & 1.8e-03 & 8.07e-05 & 2.02 & 6.43e-02 & 1.00 & 30 \\
    \hline
    \multirow{6}{*}{\bubbleboxhytable{2}} & 9.1e-02 & 1.18e+00 & - & 5.6e-02 & 7.05e-02 & - & 1.99e+00 & - & 31 \\
     & 4.8e-02 & 4.67e-01 & 1.43 & 2.8e-02 & 2.08e-02 & 1.81 & 1.02e+00 & 0.98 & 34 \\
     & 2.4e-02 & 2.17e-01 & 1.14 & 1.4e-02 & 5.39e-03 & 1.97 & 5.16e-01 & 1.00 & 33 \\
     & 1.2e-02 & 9.70e-02 & 1.18 & 7.2e-03 & 1.35e-03 & 2.01 & 2.59e-01 & 1.00 & 33 \\
     & 6.2e-03 & 4.55e-02 & 1.10 & 3.6e-03 & 3.40e-04 & 1.99 & 1.29e-01 & 1.00 & 33 \\
     & 3.1e-03 & 2.20e-02 & 1.06 & 1.8e-03 & 8.41e-05 & 2.02 & 6.46e-02 & 1.00 & 33 \\
    \hline
    \multirow{6}{*}{\bubbleboxnovtable{2}} & 9.1e-02 & 9.47e-01 & - & 5.6e-02 & 8.23e-02 & - & 1.98e+00 & - & 28 \\
     & 4.8e-02 & 3.91e-01 & 1.37 & 2.8e-02 & 2.40e-02 & 1.82 & 1.02e+00 & 0.98 & 30 \\
     & 2.4e-02 & 1.79e-01 & 1.17 & 1.4e-02 & 6.20e-03 & 1.98 & 5.13e-01 & 1.00 & 29 \\
     & 1.2e-02 & 8.09e-02 & 1.17 & 7.2e-03 & 1.56e-03 & 2.00 & 2.57e-01 & 1.00 & 29 \\
     & 6.2e-03 & 3.84e-02 & 1.08 & 3.6e-03 & 3.93e-04 & 2.00 & 1.29e-01 & 1.00 & 29 \\
     & 3.1e-03 & 1.86e-02 & 1.05 & 1.8e-03 & 9.72e-05 & 2.02 & 6.43e-02 & 1.00 & 30 \\
    \hline
    \multirow{6}{*}{\bubbleboxfemtable{2}} & 9.1e-02 & 1.55e+00 & - & 5.6e-02 & 1.31e-01 & - & 2.09e+00 & - & 32 \\
     & 4.8e-02 & 6.43e-01 & 1.36 & 2.8e-02 & 3.57e-02 & 1.92 & 1.05e+00 & 1.02 & 35 \\
     & 2.4e-02 & 3.18e-01 & 1.05 & 1.4e-02 & 9.11e-03 & 2.00 & 5.23e-01 & 1.01 & 34 \\
     & 1.2e-02 & 1.41e-01 & 1.19 & 7.2e-03 & 2.28e-03 & 2.01 & 2.61e-01 & 1.01 & 34 \\
     & 6.2e-03 & 6.55e-02 & 1.12 & 3.6e-03 & 5.69e-04 & 2.01 & 1.30e-01 & 1.00 & 34 \\
     & 3.1e-03 & 3.16e-02 & 1.06 & 1.8e-03 & 1.42e-04 & 2.01 & 6.50e-02 & 1.00 & 34 \\
    \hline
  \bottomrule
  \end{tabular}
  \end{small}
\end{table}
The numerical results are shown in \cref{tab:engelman-delaunay,tab:engelman-distorted}. As in the previous test case, 2nd order $L^2$ and first-order $H^1$ convergence rates are observed for the velocity on both grids and for all schemes. Super-convergence for the pressure is again only observed on the Delaunay grid. 
The errors of all schemes are again in the same order of magnitude and $H^1$ velocity errors are very similar. On Delaunay grids, the errors of the non-overlapping and overlapping schemes are smaller compared to those of the other schemes. On randomly distorted grids, the $L^2$ velocity errors of the non-overlapping scheme are larger than those of the hybrid and overlapping schemes. The result with the classical finite elements (MINI) results in the largest errors. Also in this test case, the number of linear solver iterations indicates the stability of the schemes. 

\subsubsection{Taylor–Green}
The periodic initial solution of the classical 3D Taylor-Green vortex benchmark is given in \cite{Boyer2017} for $\Omega = (0,1)^3$, with viscosity $\mu = 1$, $\boldsymbol{f} = (36\pi^2\cos(2\pi x)\sin(2\pi y)\sin(2\pi z),0,0)$, $\vel = (v_x, v_y, v_z)^T$ and $p$ as:
\begin{align}
    v_x(x,y,z) &= -2\cos(2\pi x)\sin(2\pi y)\sin(2\pi z), \\
    v_y(x,y,z) &= \sin(2\pi x)\cos(2\pi y)\sin(2\pi z), \\
    v_z(x,y,z) &= \sin(2\pi x)\sin(2\pi y)\cos(2\pi z), \\
    p(x,y,z) &= -6\pi \sin(2\pi x)\sin(2\pi y)\sin(2\pi z).
\end{align}
Mixed boundary conditions matching the analytical solution are set for the momentum balance, i.e.\ Neumann conditions on the right, rear, and top boundaries and Dirichlet conditions on the remaining parts of the domain boundaries.

\begin{table}[!hbtp]
\begin{small}
  \centering
  \caption{\textbf{Convergence study}. Taylor–Green test case (3D) on a Delaunay grid.}
  \label{tab:greendelaunay}
  \begin{tabular}{l*{9}{l}}
    \toprule
scheme &     $h^p$ & $\lVert p_h - p\rVert_{L^2}$ &rate & $h^{\vel}$ & $\lVert \vel_h - \vel\rVert_{L^2}$ &rate & $\lVert \vel_h - \vel\rVert_{H^1}$ &rate & it \\
    \midrule
    \multirow{5}{*}{\bubbleboxovtable{3}} & 1.6e-01 & 7.76e+00 & - & 1.0e-01 & 4.61e-01 & - & 7.40e+00 & - & 50 \\
     & 9.6e-02 & 3.46e+00 & 1.53 & 5.6e-02 & 1.18e-01 & 2.28 & 3.94e+00 & 1.06 & 57 \\
     & 7.6e-02 & 2.18e+00 & 1.97 & 4.3e-02 & 6.09e-02 & 2.54 & 2.94e+00 & 1.13 & 60 \\
     & 6.0e-02 & 1.45e+00 & 1.72 & 3.3e-02 & 3.39e-02 & 2.26 & 2.23e+00 & 1.07 & 63 \\
     & 4.2e-02 & 8.24e-01 & 1.58 & 2.3e-02 & 1.46e-02 & 2.21 & 1.49e+00 & 1.06 & 64 \\
    \hline
    \multirow{5}{*}{\bubbleboxhytable{3}} & 1.6e-01 & 8.48e+00 & - & 1.0e-01 & 4.68e-01 & - & 7.39e+00 & - & 54 \\
     & 9.6e-02 & 3.83e+00 & 1.50 & 5.6e-02 & 1.20e-01 & 2.28 & 3.95e+00 & 1.05 & 60 \\
     & 7.6e-02 & 2.42e+00 & 1.97 & 4.3e-02 & 6.19e-02 & 2.55 & 2.94e+00 & 1.13 & 63 \\
     & 6.0e-02 & 1.60e+00 & 1.73 & 3.3e-02 & 3.45e-02 & 2.26 & 2.23e+00 & 1.07 & 65 \\
     & 4.2e-02 & 9.14e-01 & 1.57 & 2.3e-02 & 1.48e-02 & 2.22 & 1.49e+00 & 1.06 & 66 \\
    \hline
    \multirow{5}{*}{\bubbleboxfemtable{3}} & 1.6e-01 & 1.15e+01 & - & 1.0e-01 & 4.57e-01 & - & 6.77e+00 & - & 65 \\
     & 9.6e-02 & 6.21e+00 & 1.16 & 5.6e-02 & 1.81e-01 & 1.55 & 4.09e+00 & 0.85 & 72 \\
     & 7.6e-02 & 4.05e+00 & 1.84 & 4.3e-02 & 1.03e-01 & 2.18 & 3.03e+00 & 1.14 & 74 \\
     & 6.0e-02 & 2.66e+00 & 1.76 & 3.3e-02 & 5.93e-02 & 2.11 & 2.29e+00 & 1.08 & 78 \\
     & 4.2e-02 & 1.53e+00 & 1.55 & 2.3e-02 & 2.67e-02 & 2.10 & 1.52e+00 & 1.08 & 78 \\    
    
    \hline
  \bottomrule
  \end{tabular}
  \end{small}
\end{table}

\begin{table}[!hbtp]
\begin{small}
  \centering
  \caption{\textbf{Convergence study}. Taylor–Green test case (3D) on randomly distorted grid.}
  \label{tab:greendistorted}
  \begin{tabular}{l*{9}{l}}
    \toprule
scheme &     $h^p$ & $\lVert p_h - p\rVert_{L^2}$ &rate & $h^{\vel}$ & $\lVert \vel_h - \vel\rVert_{L^2}$ &rate & $\lVert \vel_h - \vel\rVert_{H^1}$ &rate & it \\
    \midrule
    \multirow{5}{*}{\bubbleboxovtable{3}} & 1.7e-01 & 6.18e+00 & - & 1.0e-01 & 4.34e-01 & - & 7.21e+00 & - & 41 \\
     & 9.1e-02 & 1.95e+00 & 1.90 & 5.1e-02 & 1.36e-01 & 1.72 & 3.88e+00 & 0.91 & 44 \\
     & 7.1e-02 & 1.44e+00 & 1.25 & 4.0e-02 & 8.13e-02 & 1.99 & 3.04e+00 & 0.94 & 45 \\
     & 5.6e-02 & 9.93e-01 & 1.49 & 3.0e-02 & 5.06e-02 & 1.78 & 2.35e+00 & 0.97 & 47 \\
     & 3.8e-02 & 6.33e-01 & 1.22 & 2.1e-02 & 2.45e-02 & 1.90 & 1.62e+00 & 0.98 & 49 \\
    \hline
    \multirow{5}{*}{\bubbleboxhytable{3}} & 1.7e-01 & 6.70e+00 & - & 1.0e-01 & 4.36e-01 & - & 7.18e+00 & - & 43 \\
     & 9.1e-02 & 2.12e+00 & 1.90 & 5.1e-02 & 1.38e-01 & 1.71 & 3.89e+00 & 0.91 & 46 \\
     & 7.1e-02 & 1.58e+00 & 1.21 & 4.0e-02 & 8.23e-02 & 1.99 & 3.05e+00 & 0.94 & 48 \\
     & 5.6e-02 & 1.09e+00 & 1.47 & 3.0e-02 & 5.12e-02 & 1.78 & 2.35e+00 & 0.97 & 50 \\
     & 3.8e-02 & 7.01e-01 & 1.21 & 2.1e-02 & 2.48e-02 & 1.90 & 1.62e+00 & 0.98 & 52 \\
    \hline
    \multirow{5}{*}{\bubbleboxfemtable{3}} & 1.7e-01 & 7.15e+00 & - & 1.0e-01 & 4.95e-01 & - & 6.81e+00 & - & 46 \\
     & 9.1e-02 & 3.05e+00 & 1.41 & 5.1e-02 & 1.93e-01 & 1.40 & 3.93e+00 & 0.81 & 53 \\
     & 7.1e-02 & 2.34e+00 & 1.10 & 4.0e-02 & 1.23e-01 & 1.72 & 3.09e+00 & 0.93 & 56 \\
     & 5.6e-02 & 1.67e+00 & 1.35 & 3.0e-02 & 7.75e-02 & 1.75 & 2.39e+00 & 0.97 & 58 \\
     & 3.8e-02 & 1.13e+00 & 1.06 & 2.1e-02 & 3.80e-02 & 1.86 & 1.64e+00 & 0.99 & 62 \\
    \hline
  \bottomrule
  \end{tabular}
  \end{small}
\end{table}

For this test case, the non-overlapping scheme is omitted since its extension to 3D is not straightforward.
The numerical results are shown in \cref{tab:greendelaunay,tab:greendistorted}.
The behavior of the schemes is very similar to the previous test cases and therefore only briefly discussed. The smallest $L^2$ errors are obtained for the overlapping scheme and the largest errors for the FEM. As before, super-convergence is only observed on Delaunay grids. 
The number of linear solver iterations increases only slightly and for both CVFE schemes and the FEM.
As for the 2D test cases, the FEM solver shows a slightly higher number of linear solver iterations. Quantitatively similar results (also showing these two observations) on the linear solver iterations have been reported by \cite{Boon2022} for Darcy-Stokes problems using both finite-element and finite-volume schemes.

The $L^2$ pressure errors are quite large for all schemes, yielding a relatively poor pressure approximation. This effect is known for the MINI element, \cite{Soulaimani1987,Chamberland2010}. It does not influence the convergence rate but the magnitude of errors. Thus, this is not unique to the presented CVFE schemes but a property inherited from the MINI element. Therefore, we do not discuss this in more detail but refer the reader to \cite{Sani1981,Soulaimani1987} for further details and a more general discussion.

\subsubsection{Poiseuille pipe flow}
In this test case, pressure-driven viscous flow ($\mu = \SI{1}{\milli\pascal\s}$) with mean velocity $\bar{v} = \SI{0.01}{\m\per\s}$ through a three-dimensional cylindrical pipe with radius $R=\SI{0.1}{\m}$ and length $L=\SI{1}{\m}$ is considered. The analytical solution can be derived by transforming into cylindrical coordinates, yielding the classical Poiseuille flow solution, $p(z) = (L-z) \bar{v}8\mu / R^2$, $v(z) = 2\bar{v}(1 - (r/R)^2)$, $r \in [0,R]$, $z \in [0,L]$. In this test case, consecutive finer meshes are created by running the meshing algorithm with half the target element size. The cylindrical boundary is better approximated in each step.

At the inflow and outflow boundaries, Neumann conditions are set. At the remaining cylinder surface 
$\vel = 0$ is set, see \cref{fig:pipe3d}. 
\begin{figure}[!htb]
    \centering
    \includegraphics[width=0.6\textwidth]{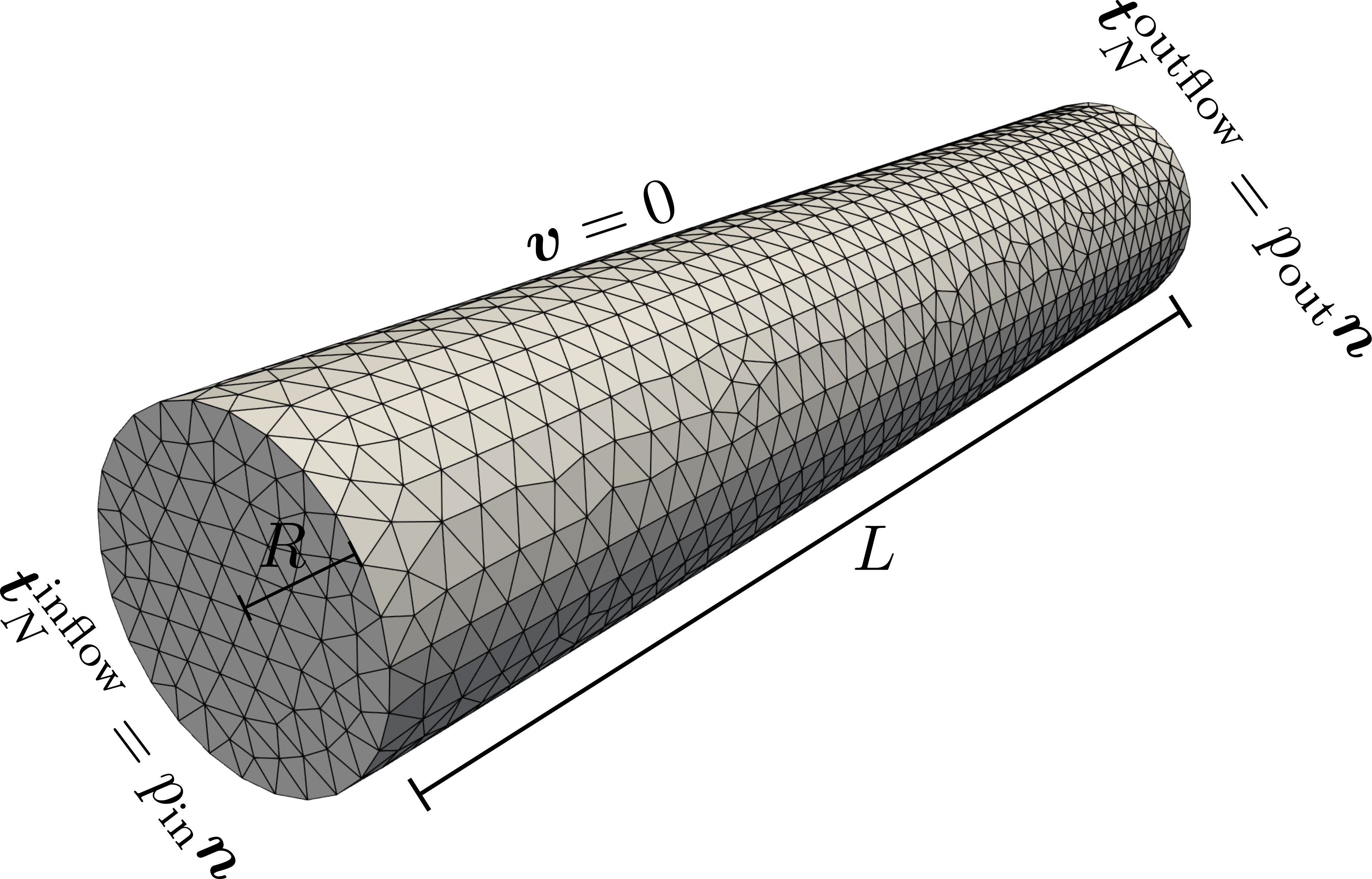}
    \caption{Boundary conditions and grid (coarse refinement level) used for the Poiseuille pipe flow test case; $R=\SI{0.1}{\m}$, $L=\SI{1}{\m}$, $\mu = \SI{1}{\milli\pascal\s}$, $p_\text{in} = \SI{0.08}{\pascal}$, $p_\text{out} = \SI{0}{\pascal}$, $\n$ denotes the outward-pointing unit normal on the cylinder surface and $\boldsymbol{t}$ the normal traction.}
    \label{fig:pipe3d}
\end{figure}

\begin{table}[!hbtp]
\begin{small}
  \centering
  \caption{\textbf{Convergence study}. Three-dimensional pipe flow using a with $\mu = \SI{1}{\milli\pascal\s}$.}
  \label{tab:pipe3d}
  \begin{tabular}{l*{9}{l}}
    \toprule
scheme &     $h^p$ & $\lVert p_h - p\rVert_{L^2}$ &rate & $h^{\vel}$ & $\lVert \vel_h - \vel\rVert_{L^2}$ &rate & $\lVert \vel_h - \vel\rVert_{H^1}$ &rate & it \\
    \midrule
    \multirow{6}{*}{\bubbleboxovtable{3}} & 4.1e-02 & 5.93e-05 & - & 2.5e-02 & 1.72e-03 & - & 7.70e-02 & - & 45 \\
     & 3.1e-02 & 3.44e-05 & 2.11 & 1.9e-02 & 9.48e-04 & 2.02 & 5.87e-02 & 0.92 & 44 \\
     & 2.3e-02 & 2.17e-05 & 1.57 & 1.3e-02 & 4.78e-04 & 2.10 & 4.12e-02 & 1.08 & 43 \\
     & 1.7e-02 & 1.35e-05 & 1.48 & 9.5e-03 & 2.31e-04 & 2.07 & 2.85e-02 & 1.05 & 43 \\
     & 1.2e-02 & 9.05e-06 & 1.27 & 6.8e-03 & 1.17e-04 & 2.02 & 2.02e-02 & 1.02 & 44 \\
     & 8.9e-03 & 5.93e-06 & 1.29 & 4.8e-03 & 5.80e-05 & 2.04 & 1.41e-02 & 1.03 & 45 \\
    \hline
    \multirow{6}{*}{\bubbleboxhytable{3}} & 4.1e-02 & 6.54e-05 & - & 2.5e-02 & 1.71e-03 & - & 7.72e-02 & - & 47 \\
     & 3.1e-02 & 3.80e-05 & 2.11 & 1.9e-02 & 9.44e-04 & 2.03 & 5.88e-02 & 0.92 & 45 \\
     & 2.3e-02 & 2.40e-05 & 1.57 & 1.3e-02 & 4.75e-04 & 2.11 & 4.13e-02 & 1.08 & 44 \\
     & 1.7e-02 & 1.49e-05 & 1.49 & 9.5e-03 & 2.29e-04 & 2.08 & 2.85e-02 & 1.05 & 45 \\
     & 1.2e-02 & 9.95e-06 & 1.27 & 6.8e-03 & 1.16e-04 & 2.02 & 2.02e-02 & 1.02 & 45 \\
     & 8.9e-03 & 6.52e-06 & 1.29 & 4.8e-03 & 5.75e-05 & 2.04 & 1.42e-02 & 1.03 & 46 \\
    \hline
    \multirow{6}{*}{\bubbleboxfemtable{3}} & 4.1e-02 & 1.21e-04 & - & 2.5e-02 & 1.62e-03 & - & 8.00e-02 & - & 50 \\
     & 3.1e-02 & 6.74e-05 & 2.27 & 1.9e-02 & 8.77e-04 & 2.08 & 6.01e-02 & 0.97 & 51 \\
     & 2.3e-02 & 4.16e-05 & 1.64 & 1.3e-02 & 4.34e-04 & 2.16 & 4.21e-02 & 1.09 & 49 \\
     & 1.7e-02 & 2.51e-05 & 1.57 & 9.5e-03 & 2.08e-04 & 2.09 & 2.90e-02 & 1.06 & 49 \\
     & 1.2e-02 & 1.67e-05 & 1.29 & 6.8e-03 & 1.05e-04 & 2.04 & 2.05e-02 & 1.03 & 50 \\
     & 8.9e-03 & 1.08e-05 & 1.33 & 4.8e-03 & 5.17e-05 & 2.05 & 1.44e-02 & 1.04 & 51 \\ 
    
    \hline
  \bottomrule
  \end{tabular}
  \end{small}
\end{table}

The results for the three-dimensional pipe flow test are shown in \cref{tab:pipe3d}. As expected, 2nd-order $L^2$ and first-order $H^1$ convergence rates are observed for the velocity for all schemes. Super-convergence for pressure is no longer observed. This could be caused by changing boundary conditions for the mass balance equation due to the geometrical approximation of the real cylindrical domain, i.e.\ the domain boundary changes with grid refinement.

\subsubsection{Summary}
The results of the previous test cases are very similar to those that were conducted for the MINI element in previous works, see e.g.~\cite{Cioncolini2019}, i.e.\ 2nd order for $L^2$ velocity errors, 1st order for $H^1$ velocity errors, and super-convergence $\mathcal{O}(h^{3/2})$ for pressure on Delaunay meshes. For all considered test cases, the CVFE schemes yield smaller errors than the FE scheme. 
Furthermore, the number of linear solver iterations strongly indicates the stability of the presented schemes for the given test cases and the used grids. 

As mentioned before, it has been shown for the MINI element that  the linear part of the velocity yields better approximations such that the bubble part is only needed for stabilization. This has not been observed for the schemes presented in this work, which is why the results have not been shown. This is probably strongly related to the modified basis functions, as discussed in \cref{subsec:DiscSpaces}.

\subsection{Viscous flow in complex domains}
To demonstrate the robustness and conservation properties
of the proposed methods, we simulate two showcases in complex domains using the overlapping CVFE scheme $\bubbleboxname{d}{ov}$. The results of a simulation of Stokes flow around a turtle-shaped obstacle are shown in \cref{fig:flowobstacle}. The method appears stable even in the presence of sharp corner singularities and the locally refined mesh. 

The results of a simulation of Stokes flow through a blood vessel bifurcation are shown in \cref{fig:vessel}. As a sanity check, we demonstrate the grid convergence of the integral inflow rate (results shown in \cref{tab:vessel_refine_balance}) and computed the mass balance in four subdomains (results shown in \cref{tab:vessel_mass_balance}) demonstrating that the scheme is locally conservative up to numerical machine precision.

\begin{figure}[!htb]
    \centering
    \includegraphics[height=0.3\textwidth]{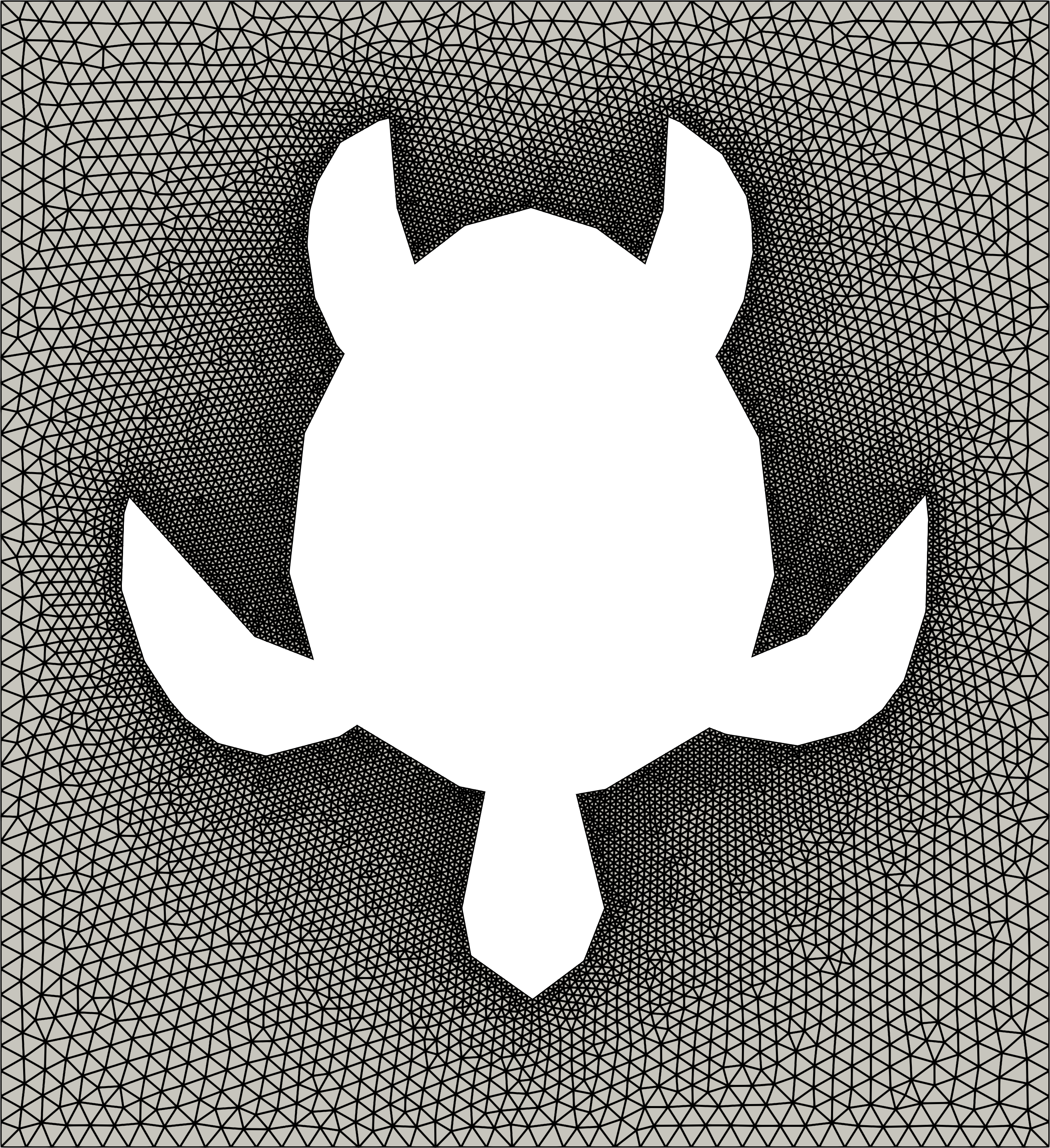} \hfill
    \includegraphics[height=0.3\textwidth]{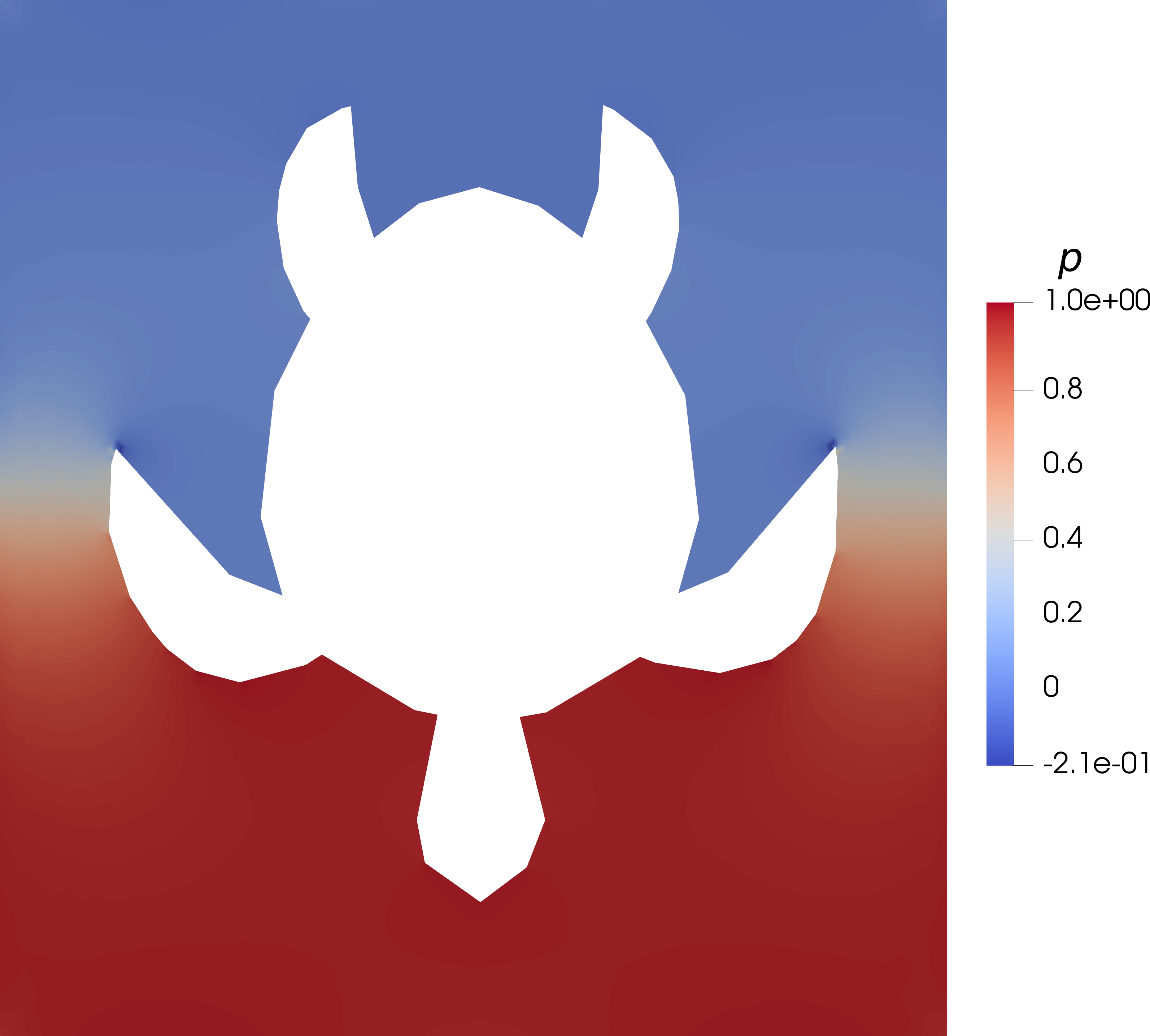} \hfill
    \includegraphics[height=0.3\textwidth]{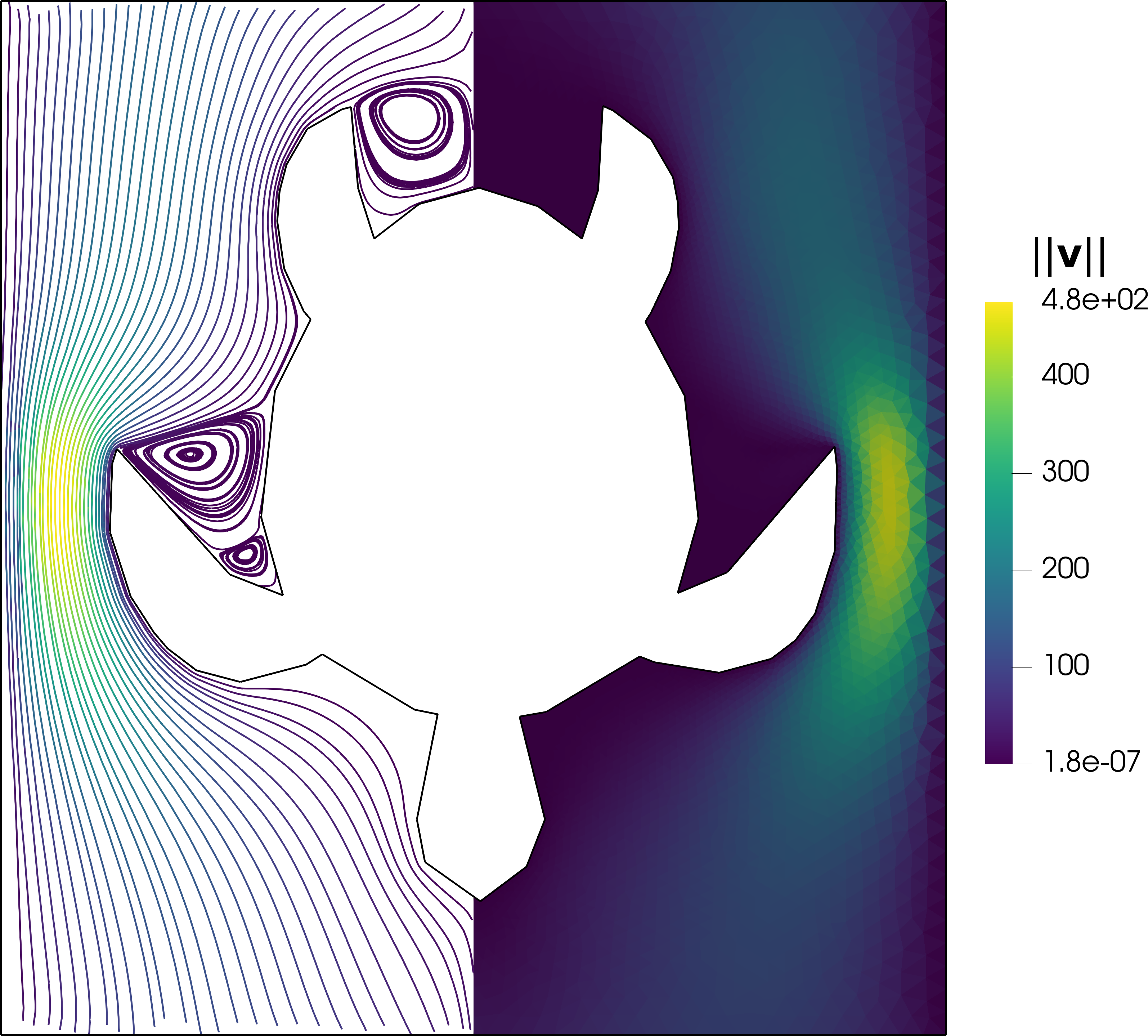}
    \caption{\textbf{Stokes flow around a turtle-shaped obstacle}. Grid used for the simulations (left), pressure (middle), and velocity solutions (right). The pressure approximation is physically consistent at the corner singularities. Moffatt eddies behind the extremities become obvious in the streamline pattern.}
    \label{fig:flowobstacle}
\end{figure}

\begin{figure}[!htb]
    \centering
    \includegraphics[width=\textwidth]{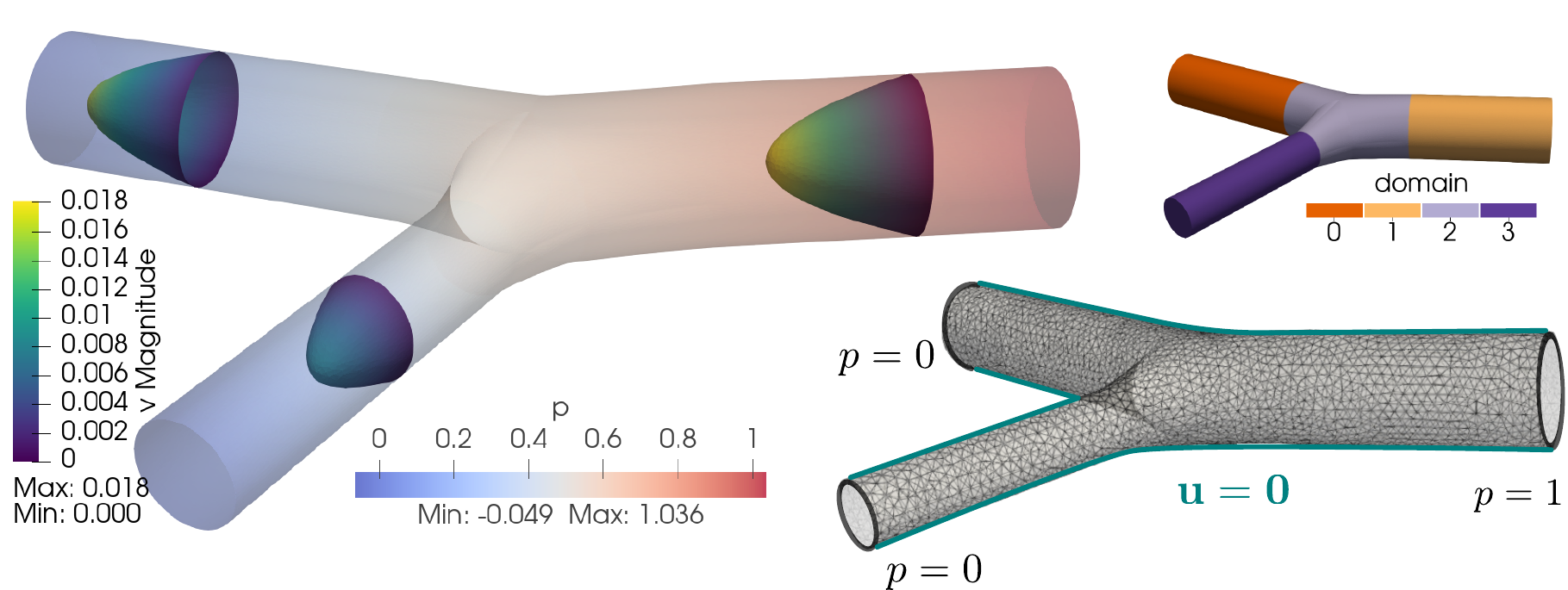}
    \caption{\textbf{Stokes flow through a vessel bifurcation}. The domain regions are used to demonstrate the mass conservation property of the discretization scheme. The pressure at inflow and outflow cross-sections is weakly enforced (natural boundary condition). Zero velocity is strongly enforced on all side walls. Velocity profiles are visualized on three slices by deforming the slice surface with the velocity vector field. The shown mesh corresponds to a problem with $120$k degrees of freedom.}
    \label{fig:vessel}
\end{figure}

\begin{table}[!htb]
    \centering
    \caption{\textbf{Grid convergence of inflow rate and global mass balance.} Integrated inflow rate and global mass balance with grid refinement for the vessel bifurcation case. The global mass balance is obtained by integrating $\rho \boldsymbol{v}_h\cdot \boldsymbol{n}$ over all boundaries ($\rho = \SI{1000}{\kg\per\cubic\m}$). It is normalized to the inflow rate. The column $\Delta Q$ shows the difference of the inflow rate value to the value on the next finer grid. The difference decreases with refinement indicating convergence and from the penultimate to the ultimate refinement $Q$ changes by less than \SI{1}{\percent}.}
    \label{tab:vessel_refine_balance}
    \begin{tabular}{r c c c} \toprule
         \#dofs & inflow rate $Q$ & $\Delta Q$ & global mass balance \\ \midrule
         17144 & \SI{2.3942e-5}{\kg\per\s} & \SI{30.6e-7}{\kg\per\s} & \SI{-16e-19}{\kg\per\s} \\
         114953 & \SI{2.7005e-5}{\kg\per\s} & \SI{8.5e-7}{\kg\per\s} & \SI{-16e-19}{\kg\per\s} \\
         756855 & \SI{2.7855e-5}{\kg\per\s} & \SI{2.4e-7}{\kg\per\s} & \SI{-4e-19}{\kg\per\s} \\
         5290409 & \SI{2.8090e-5}{\kg\per\s} & - & \SI{5e-19}{\kg\per\s} \\
    \bottomrule
    \end{tabular}
\end{table}

\begin{table}[!htb]
    \centering
    \caption{\textbf{Local mass balance in subregions.} The mass balance for the $4$ non-overlapping regions shown in \cref{fig:flowobstacle} (computed on a tetrahedral grid corresponding to $~5.3$M degrees of freedom) are exact up to machine precision since the flux approximation scheme is locally mass conservative and the regions are chosen as the union of all vertex-centered control volumes centered in the region.}
    \label{tab:vessel_mass_balance}
    \begin{tabular}{l c c c c} \toprule
         region & 0 & 1 & 2 & 3 \\ \midrule
         mass balance & \SI{-2e-19}{\kg\per\s} & \SI{8e-19}{\kg\per\s} & \SI{7e-20}{\kg\per\s} & \SI{-2e-19}{\kg\per\s} \\
    \bottomrule
    \end{tabular}
\end{table}




\section{Summary and conclusion}
Three types of control-volume finite-element schemes, namely non-overlapping, overlapping, and hybrid, have been presented and investigated on the basis of the inf-sup stable MINI finite element. All of the presented approaches are locally mass and momentum conservative for a given (sub-)set of control volumes. The overlapping and hybrid schemes have the advantage that both mass and momentum are conserved on identical control volumes. These control volumes form a partition of the domain.
The hybrid scheme generalizes to higher-order polynomial approximations of the solution while mass and momentum are conserved on the same control volumes as in the lowest-order case. As a caveat, due to the inherent relation to the MINI finite element, all schemes lack pressure robustness (a property of many common Stokes finite elements as discussed in detail in \citep{John2017}).

The discretizations are structure-preserving in the sense that they preserve the local conservation property strongly (conservation of mass and conservation of linear momentum) which is fundamental to the Stokes equations as a set of continuity equations. At the same time, all presented schemes appear to be stable in numerical experiments, including grid convergence tests and the observation of boundedness of preconditioned Krylov solver iterations (indicating bounded singular values of the preconditioned discrete Stokes operator).

The presented schemes use classical vertex-centered finite-volume schemes~\citep{Winslow1966,Bank1987Box} as a building block. In this way, they appear to be convenient for coupled problems such as the (Navier-)Stokes-Darcy problem~\citep{schneider2021coupling} for which the pressure variable can then be discretized uniformly in the entire domain. At the same time, local mass and momentum conservation across the coupling interface can be achieved. Such coupled systems will be the subject of upcoming work.

\section*{Acknowledgements}
The authors would like to thank Kent-André Mardal and Rainer Helmig for constructive discussions. TK acknowledges financial support from the European Union’s Horizon 2020 Research and Innovation programme under the Marie Skłodowska‐Curie Actions Grant agreement No 801133. MS would like to thank the German Research Foundation (DFG) for supporting this work by funding SFB 1313, Project Number
327154368, Research Project A02.

\appendix

\section{Analytical solution for 2D Donea-Huerta}
\label{sec:appendix:donea}
By method of manufactured solutions, the following analytical solution for $\Omega = \left(0, 1\right)^2$ solves the Stokes equations, \cref{eq:stokes}, with $\vel = (v_x, v_y)^T$,
\begin{align}
    v_x(x,y) &= h(x) h^\prime(y), \\
    v_y(x,y) &= -h(y) h^\prime(x) \\
    p(x,y) &= g(x),
\end{align}
where $g(u) := u(1-u)$; $h(u) := g^2(u) = u^2(1-u)^2$; $h^\prime(u) := \frac{\partial h}{\partial u} = 2u - 6u^2 + 4u^3$, for a given right-hand-side, $\boldsymbol{f} = (f_x, f_y)^T$,
\begin{align}
    f_x(x,y) &= -2\mu \partial_{xx}\vel_x - \mu\partial_{yy}\vel_x - \mu\partial_{xy}\vel_y + \partial_{x}p, \\
    f_y(x,y) &= -2\mu \partial_{yy}\vel_y - \mu\partial_{xx}\vel_y - \mu\partial_{xy}\vel_x + \partial_{y}p,
\end{align}
where with $h^{\prime\prime}(u) := \frac{\partial^2 h}{\partial u^2} = 2 - 12u + 12u^2$; $h^{\prime\prime\prime}(u) := \frac{\partial^3 h}{\partial u^3} = -12 + 24u$,
\begin{align*}
    \partial_{xx}\vel_x &= h^{\prime\prime}(x) h^{\prime}(y), & \partial_{yy}\vel_x &= h(x) h^{\prime\prime\prime}(y), & \partial_{xy}\vel_x &= h^{\prime}(x) h^{\prime\prime}(y), \\ 
    \partial_{yy}\vel_y &= -h^{\prime\prime}(y) h^{\prime}(x), & \partial_{xx}\vel_y &= -h(y) h^{\prime\prime\prime}(x), & \partial_{xy}\vel_y &= -h^{\prime}(y) h^{\prime\prime}(x) \\
    \partial_{x}p &= 1 - 2x, & \partial_{y}p &= 0. &&
\end{align*}
The solution corresponds to the one presented by \cite{Donea2003} trivially extended for arbitrary viscosity, $\mu$.

 \bibliographystyle{elsarticle-num} 
 \bibliography{2023_cvfe_mini}





\end{document}